\documentclass[reqno]{amsart}
\usepackage{amssymb,latexsym,amsmath,amsthm,enumerate,pifont}
\usepackage[mathscr]{eucal}
\usepackage{framed}
\usepackage{mathrsfs}
\usepackage[all]{xy}
\usepackage{tikz}
\usetikzlibrary{positioning}

\DeclareSymbolFont{bbold}{U}{bbold}{m}{n}
\DeclareSymbolFontAlphabet{\mathbbold}{bbold}

\theoremstyle{plain}
\newtheorem{thm}{Theorem}
\newtheorem{thmAB}[thm]{Theorem}
\newtheorem{thmBZ}[thm]{Theorem}
\newtheorem{lem}[thm]{Lemma}

\theoremstyle{definition}

\newcommand{\con}{\mathsf{con}}
\newcommand{\simp}{\mathsf{sim}}
\newcommand{\non}{\mathsf{non}}
\newcommand{\h}{\mathsf{h}}
\newcommand{\tail}{\mathsf{t}}

\newcommand{\up}[1]{\textup{#1}}

\newcommand{\dsharp}{\raisebox{.01em}{\kern-.08em\scalebox{.6}{\rotatebox[origin=c]{45}{\ding{67}}}}\kern-.017em}

\newcommand{\girth}{\operatorname{girth}}

\newcommand{\be}{\mathbf{e}}

\newcommand{\bn}{\mathbf{m}}
\newcommand{\bm}{\mathbf{n}}

\newcommand{\bp}{\mathbf{p}}
\newcommand{\bq}{\mathbf{q}}

\newcommand{\bt}{\mathbf{t}}
\newcommand{\bu}{\mathbf{u}}
\newcommand{\bv}{\mathbf{v}}
\newcommand{\bw}{\mathbf{w}}

\newcommand{\bz}{\mathbf{z}}

\newcommand{\sa}{\mathsf{a}}
\newcommand{\ssb}{\mathsf{b}}
\newcommand{\ssc}{\mathsf{c}}
\newcommand{\sd}{\mathsf{d}}
\newcommand{\se}{\mathsf{e}}
\newcommand{\sh}{\mathsf{h}}

\newcommand{\st}{\mathsf{t}}

\begin{document}

\title[From $A$ to $B$ to $Z$]{From $A$ to $B$ to $Z$}

\author[Marcel Jackson and Wen Ting Zhang]{Marcel Jackson$^1$ and Wen Ting Zhang$^2$}
\address{$^{1,2}$Department of Mathematics and Statistics, La Trobe University, Victoria  3086,
Australia}\address{
$^2$School of Mathematics and Statistics, Lanzhou University, Lanzhou, Gansu 730000, PR China} \email{M.G.Jackson@latrobe.edu.au}\email{zhangwt@lzu.edu.cn}

\thanks{The first author was supported by ARC Future Fellowship FT120100666 and the second author by ARC Discovery Project DP1094578, the National Natural Science Foundation of China (nos.~11771191, 11401275, 11371177) and the Natural
Science Foundation of Gansu Province (no. 20JR5RA275)}

\begin{abstract}
The variety generated by the Brandt semigroup ${\bf B}_2$ can be defined within the variety generated by the semigroup ${\bf A}_2$ by the single identity $x^2y^2\approx y^2x^2$.  Edmond Lee asked whether or not the same is true for the monoids ${\bf B}_2^1$ and ${\bf A}_2^1$.  We employ an encoding of the homomorphism theory of hypergraphs to show that there is in fact a continuum of distinct subvarieties of ${\bf A}_2^1$ that satisfy $x^2y^2\approx y^2x^2$ and contain ${\bf B}_2^1$.  A further consequence is that the variety of ${\bf B}_2^1$ cannot be defined within the variety of ${\bf A}_2^1$ by any finite system of identities. Continuing downward, we then turn to subvarieties of ${\bf B}_2^1$.  We resolve part of a further question of Lee by showing that there is a continuum of distinct subvarieties all satisfying the stronger identity $x^2y\approx yx^2$ and containing the monoid $M(\bz_\infty)$, where $\bz_\infty$ denotes the infinite limit of the Zimin words $\bz_0=x_0$, $\bz_{n+1}=\bz_n x_{n+1}\bz_n$.
\end{abstract}

\maketitle

The six element Brandt monoid ${\bf B}_2^1$ can be represented as a matrix semigroup
\[
\begin{tabular}{cccccc}
$\left(\begin{matrix} 0&0\\0&0\end{matrix}\right)$
&
$\left(\begin{matrix} 1&0\\0&1\end{matrix}\right)$
&
$\left(\begin{matrix} 0&1\\0&0\end{matrix}\right)$
&
$\left(\begin{matrix} 0&0\\1&0\end{matrix}\right)$
&
$\left(\begin{matrix} 1&0\\0&0\end{matrix}\right)$
&
$\left(\begin{matrix} 0&0\\0&1\end{matrix}\right)$
\\
\rule{0cm}{.5cm}$0$&$1$&$\sa$&$\ssb$&$\sa\ssb$&$\ssb\sa$
\end{tabular}
\]
or as a monoid with presentation $\langle \sa,\ssb\mid \sa\ssb\sa=\sa, \ssb\sa\ssb=\ssb,\sa\sa=\ssb\ssb=0\rangle$.  This Brandt monoid is perhaps the most ubiquitous harbinger of complex behaviour in all finite semigroups.  It has no finite basis for its identities, and not only was the first known example \cite{per}, but is the equal smallest generator for a semigroup variety with this property \cite{leezha}.  More than being just nonfinitely based, it is \emph{inherently nonfinitely based} in the sense that every locally finite variety containing it is without a finite identity basis \cite{sap}.  It is also the equal smallest generator with this property, and moreover a regular semigroup is inherently nonfinitely based if and only if it generates a variety containing ${\bf B}_2^1$ \cite{jac:INFB}.  It is the smallest known semigroup with co-\texttt{NP}-complete identity checking problem \cite{kli,sei}, the only known nonfinitely related semigroup~\cite{may}, and the smallest known (and smallest possible) semigroup with \texttt{NP}-hard membership problem for its variety \cite{jck}.  It is the smallest known generator for a semigroup variety with continuum many semigroup subvarieties \cite{ELL,jac00} and for a monoid variety with continuum many monoid subvarieties \cite{jaclee}.

The semigroup ${\bf A}_2^1$ is the slightly less glamorous sibling of ${\bf B}_2^1$.  It can be represented as the  matrix semigroup
\[
\begin{tabular}{cccccc}
$\left(\begin{matrix} 0&0\\0&0\end{matrix}\right)$
&
$\left(\begin{matrix} 1&0\\0&1\end{matrix}\right)$
&
$\left(\begin{matrix} 0&1\\0&0\end{matrix}\right)$
&
$\left(\begin{matrix} 1&0\\1&0\end{matrix}\right)$
&
$\left(\begin{matrix} 1&0\\0&0\end{matrix}\right)$
&
$\left(\begin{matrix} 0&1\\0&1\end{matrix}\right)$
\\
\rule{0cm}{.5cm}$0$&$1$&$\ssc$&$\sd$&$\ssc\sd$&$\sd\ssc$
\end{tabular}
\]
or as the monoid with presentation $\langle \ssc,\sd\mid \ssc\sd\ssc=\ssc, \sd\ssc\sd=\sd,\ssc\ssc=0,\sd\sd=\sd\rangle$.
While~${\bf A}_2^1$ plays a similarly critical role in the structural complexity of semigroups, it generates a variety properly containing that generated by ${\bf B}_2^1$, and so some of its bad properties are simply inherited from ${\bf B}_2^1$.
On other properties it is better behaved: identity checking on ${\bf A}_2^1$ can be performed in polynomial time for example~\cite{kli}.  Nevertheless, being so small gives it some undeniable charisma: for example, a finite regular semigroup is inherently nonfinitely based if and only if either ${\bf B}_2^1$ or~${\bf A}_2^1$ divides it~\cite{jac:INFB}.

The third character in our story is the mysterious semigroup $S(\bz_\infty)$ and its monoid variant  $M(\bz_\infty)$.  The word~$\bz_\infty$ is the right infinite limit of the following sequence: $\bz_0:=x_0$, $\bz_{n+1}:=\bz_nx_{n+1}\bz_n$, where $x_0,x_1,\dots$ is an infinite list of pairwise distinct letters.  The semigroup $S(\bz_\infty)$ has as its elements, a zero element along with all finite subwords of $\bz_\infty$, with the product $\bu\cdot\bv:=\bu\bv$, if $\bu\bv$ is a subword of $\bz_\infty$, and $0$ otherwise.  Evidently $S(\bz_\infty)$ is infinite, so in terms of criticality in the theory of finite semigroups might appear to be at a disadvantage when compared with ${\bf B}_2^1$ and ${\bf A}_2^1$.  Yet $S(\bz_\infty)$ generates a semigroup variety that lies within the variety of many finite semigroups, and moreover uniquely determines at least one major finiteness condition in finite semigroups: a finite semigroup is inherently nonfinitely based if and only if it contains the semigroup $S(\bz_\infty)$ in its variety \cite{sap}.  In fact, $S(\bz_\infty)$ generates the unique minimal inherently nonfinitely based semigroup variety amongst those varieties not containing inherently nonfinitely based groups (the existence of these remains unresolved) \cite{sap2}.  In particular though, $S(\bz_\infty)$ generates a proper subvariety of the semigroup variety $\mathbb{V}_s({\bf B}_2^1)$, and hence also of $\mathbb{V}_s({\bf A}_2^1)$.

Throughout, we will use $[\![\Sigma]\!]$ to denote the variety defined by an identity or set of identities $\Sigma$, which will be either a monoid variety or semigroup variety depending on the context (typically a monoid variety). We will use the notation $\mathbb{V}_m({\bf M})$ to denote the monoid variety generated by a monoid ${\bf M}$ though we use the notation $\mathbb{B}_2^1$ and $\mathbb{A}_2^1$ for the specific varieties $\mathbb{V}_m(\mathbf{B}_2^1)$ and $\mathbb{V}_m(\mathbf{A}_2^1)$ respectively.  We let $\mathbb{V}_s({\bf S})$ denote the semigroup variety generated by a semigroup ${\bf S}$, noting that in the case of a monoid ${\bf M}$, many of the aforementioned properties hold equally for $\mathbb{V}_m({\bf M})$ and $\mathbb{V}_s({\bf M})$, with the exception of variety lattices.  The lattice of semigroup subvarieties of $\mathbb{V}_s({\bf M})$  can in some cases be enormously larger than the lattice of monoid subvarieties of $\mathbb{V}_m({\bf M})$~\cite[Lemma 4.4]{jac05}, and it is substantially more difficult to establish complex properties of variety lattices in the monoid setting in comparison to the semigroup setting; \cite{guslee,gusver,jaclee}.  The current article will establish all results in the monoid setting, with the corresponding results for semigroup varieties being immediate consequences.

In the absence of an identity element, the five element semigroups ${\bf B}_2$ and ${\bf A}_2$ are also important critical objects in the study of finite semigroups, where they relate to the possible structure of regular $\mathscr{D}$-classes.  They have finitely based equational theories, and moreover, the semigroup variety generated by ${\bf B}_2$ is defined within that generated by ${\bf A}_2$ by the single extra identity $x^2y^2\approx y^2x^2$; see Trahtman \cite{tra} or Lee and Volkov \cite[Remark~3.1]{leevol}.  This and other evidence led Edmond Lee \cite[Question 7.2, law 7.1]{lee} to ask whether the same law defines the monoid variety~$\mathbb{B}_2^1$ within the monoid variety $\mathbb{A}_2^1$.  A further question is \cite[Question 7.3]{lee}, which asks whether every monoid variety in the interval $[\mathbb{B}_2^1,\mathbb{A}_2^1]$ can be defined by a finite system of identities.    This second problem additionally arises in the work of the first author \cite{jck}, where it is shown that for any finite completely 0-simple semigroup $C$, the pseudovariety generated by the monoid $C^1$ has \texttt{NP}-hard membership problem provided it contains ${\bf B}_2^1$ but not ${\bf A}_2^1$.  If $\mathbb{B}_2^1$ can be defined within $\mathbb{A}_2^1$ by a finite system of identities, then these intractability results would push even higher.  The first main result of the present article is the following theorem which shows that the answer to all of these questions is negative. 
\begin{thmAB}[From A to B]\label{thm:AB}
\begin{enumerate}
\item There are continuum many monoid varieties in the interval $[\mathbb{B}_2^1,[\![x^2y^2\approx y^2x^2]\!]\wedge \mathbb{A}_2^1]$.
\item The monoid variety $\mathbb{B}_2^1$ is not defined within $\mathbb{A}_2^1$ by any finite system of identities.
\end{enumerate}
\end{thmAB}
This theorem, and the structures used for its proof, corroborate the suggestion in \cite[pp.~368]{lee} that it is ``extremely difficult, if not impossible, to identify all varieties in the interval $[\mathbb{B}_2^1,\mathbb{A}_2^1]$''.  The complexity of the problem of deciding membership for finite semigroups in $\mathbb{A}_2^1$ remains a very interesting unresolved problem \cite[Problem~7.5]{jck}.

We then turn our attention to the interval between $\mathbb{V}_s(S(\bz_\infty))$ and $\mathbb{V}_s({\bf B}_2^1)$ and its monoid analogue.  This interval of the lattice of semigroup/monoid varieties is particularly intriguing.  While the finitely generated varieties above $\mathbb{V}_s(S(\bz_\infty))$ are exactly the inherently nonfinitely based ones, the variety $\mathbb{V}_s({\bf B}_2^1)$ is nevertheless a minimal example amongst those that are \emph{finitely generated} \cite[Theorem 2]{sap2}.  There is certainly at least one distinct variety between $\mathbb{V}_s(S(\bz_\infty))$ and $\mathbb{V}_s({\bf B}_2^1)$, because the monoid $M(\bz_\infty)$ obtained from $S(\bz_\infty)$ by adjoining an identity element also lies in the variety generated by any finite inherently nonfinitely based semigroup.  The monoid $M(\bz_\infty)$ fails laws such as $xyz^2\approx yxz^2$, which holds on $S(\bz_\infty)$, so it generates a strictly larger semigroup variety.  A third question of Edmond Lee \cite[Question 7.5]{lee} concerns the possible strictness of the inclusions $\mathbb{M}(\bz_\infty)\subseteq \mathbb{B}_2^1\wedge [\![x^2y\approx yx^2]\!]\subseteq \mathbb{A}_2^1\wedge [\![x^2y\approx yx^2]\!]$, where $\mathbb{M}(\bz_\infty)$ denotes $\mathbb{V}_m(M(\bz_\infty))$.  The second main result of this article is the following theorem, which strongly resolves the first inclusion in \cite[Question 7.5]{lee}.
\begin{thmBZ}[From B to Z]\label{thm:BZ}
\begin{enumerate}
\item There are continuum many monoid varieties in the interval $[\mathbb{M}(\bz_\infty),[\![x^2y\approx yx^2]\!]\wedge \mathbb{B}_2^1]$.
\item The monoid variety $\mathbb{M}(\bz_\infty)$ is not defined within $\mathbb{B}_2^1$ by any finite system of identities.
\end{enumerate}
\end{thmBZ}

While continuum is the largest possible size for an interval in a subvariety lattice, we feel that in both Theorem \ref{thm:AB} and Theorem \ref{thm:BZ}, the proofs and constructions are very subtle.  The intermediate varieties we construct between $\mathbb{A}_2^1\wedge [\![x^2y^2\approx y^2x^2]\!]$ and $\mathbb{B}_2^1$ are generated by semigroups constructed by way of complicated combinatorial objects, constructed by Erd\H{os} and Hajnal using probabilistic methods.  The varieties we construct between $\mathbb{B}_2^1$ and $\mathbb{M}(\bz_\infty)$ have less opaque origins, yet a large number of failed attempts by the authors initially made speculations about the possibility equality of  $\mathbb{M}(\bz_\infty)$ and $[\![x^2y\approx yx^2]\!]\wedge \mathbb{B}_2^1$ seem plausible, even if it these were eventually found to be distinct.

\section{Preliminaries: hypergraphs}
A \emph{hypergraph} $\mathcal{H}$ is a pair $(V,E)$, where $V$ is a set and $E$ is a set of nonempty subsets of $V$.   For $k\geq 1$, the hypergraph $\mathcal{H}$ is \emph{$k$-uniform} if the elements of $E$ all have exactly $k$ elements.
The class of $2$-uniform hypergraphs coincides with the class of simple graphs, and many concepts from simple graphs extend to higher values of $k$ in obvious ways.  In particular, we say that a point $v\in V$ is \emph{isolated} if it does not lie in any hyperedge, and that a $n$-cycle is a sequence $v_0,e_0,v_1,e_1,\dots,v_{n-1},e_{n-1}$ alternating between vertices $v_0,\dots,v_{n-1}$ and hyperedges $e_0,\dots,e_{n-1}$, with no repeats and such that $v_{i+1}\in e_i\cap e_{i+1}$ (with addition in the subscript taken modulo~$n$).

The \emph{girth} of a hypergraph $\mathcal{H}$, denoted $\girth(\mathcal{H})$ is the length of the shortest cycle in $\mathcal{H}$.
We consider $3$-uniform hypergraphs with no isolated points and of girth at least $4$.  The following easy lemma recasts these girth assumptions in an equivalent form aimed at the needs of our construction.
\begin{lem}\label{lem:girth}
Let $\mathcal{H}=(V,E)$ be a $3$-uniform hypergraph.  Condition \up{(I)} below is equivalent to $\girth(\mathcal{H})\geq 3$, while conditions \up{(I)} and \up{(II)} together are equivalent to $\girth(\mathcal{H})\geq 4$.
\begin{itemize}
\item[(I)] every two-element subset $\{u,v\}\subseteq V$ is a subset of at most one hyperedge.
\item[(II)] if all two-element subsets of a triple $\{u,v,w\}\subseteq V$ extend to hyperedges in $E$, then $\{u,v,w\}$ is a hyperedge.
\end{itemize}
Moreover, conditions \up{(I)} and \up{(II)} imply the following condition\up:
\begin{itemize}
\item[(III)] Every $4$-set $\{v_1,v_2,v_3,v_4\}$ of vertices contains a two-element subset not extending to a hyperedge of $\mathcal{H}$.
\end{itemize}
\end{lem}
\begin{proof}
It is easy to see that a 2-cycle corresponds to a pair of hyperedges $\{u,v,w\}$ and $\{u,v,w'\}$ in $E$ with $w\neq w'$, which is identical to what it means to fail (I).  Thus (I) is equivalent to having no $2$-cycles.

Now assume $(V,E)$ has no $3$-cycles and assume that all two-element subsets of $\{u,v,w\}\subseteq V$ extend to hyperedges.  So, there are $u',v',w'$ such that
\[
\{u,v,w'\},\{u,v',w\},\{u',v,w\}\in E.
\]  As $(V,E)$ avoids  $3$-cycles, and $u,v,w$ avoids repeats, the sequence
\[
u,\{u,v,w'\},v,\{u',v,w\},w,\{u,v',w\}
\] must repeat in its hyperedges. This then gives one of $u=u'$, $v=v'$ or $w=w'$ and then that $\{u,v,w\}\in E$, so that (II) holds.

The converse needs (I) (which we already know forces there to be no 2-cycles).  Assume that (I) and (II) hold and consider, for contradiction, a $3$-cycle $u,e_1,v,e_2,w,e_3$.  Now, $\{u,v\}\subseteq e_1$, $\{v,w\}\subset e_2$ and $\{u,w\}\subseteq e_3$, so that $\{u,v,w\}$ is a hyperedge by~(II).  But then condition (I) forces $e_1=e_2=e_3=\{u,v,w\}$, contradicting the fact that $e_1,e_2,e_3$ are pairwise distinct.

Finally, let us assume that (I) and (II) hold and consider a $4$-element subset $\{v_1,v_2,v_3,v_4\}$ of $V$.  By (I) we cannot have both $\{v_1,v_2,v_3\}$ and $\{v_1,v_2,v_4\}$ hyperedges.  But then the contrapositive of (II) implies that there is a $2$-element subset of $\{v_1,v_2,v_3\}$ or $\{v_1,v_2,v_4\}$ (and in particular, of $\{v_1,v_2,v_3,v_4\}$) that does not extend to a hyperedge.  Thus (III) holds.
\end{proof}

\section{From $\mathbb{A}_2^1$ to $\mathbb{B}_2^1$ via $x^2y^2\approx y^2x^2$}
In this section we begin our preparations for the proof of Theorem \ref{thm:AB}.  We construct a countable family $\{M_i\mid i\in \mathbb{N}\}$ of monoids based over hypergraphs and use these to prove the result by way of the following proof structure.
\begin{enumerate}
\item ${\bf B}_2^1$ embeds into each $M_i$ (Lemma \ref{lem:B21in}); no $M_i$ lies within $\mathbb{B}_2^1$ (Lemma \ref{lem:B21}).
\item Each $M_i$ satisfies $x^2y^2\approx y^2x^2$ (Lemma \ref{lem:commute}) and lies within $\mathbb{A}_2^1$ (Lemma \ref{lem:A21}).
\item For each $i$ we find an identity ${\bf p}_i\approx {\bf q}_i$ such that $M_j\models {\bf p}_i\approx {\bf q}_i$ if and only if $i\neq j$ (Lemma \ref{lem:homommon}).
\end{enumerate}
The first two points show that for each set $P\subseteq \mathbb{N}$ the class $\{M_i\mid i\in P\}$ generates a monoid variety that is intermediate between $\mathbb{B}_2^1$ and  $\mathbb{A}_2^1\wedge [\![x^2y^2\approx y^2x^2]\!]$.  These are established in the current section.  The last point is established in the subsequent section and shows that if $P,Q\subseteq \mathbb{N}$ and $P\neq Q$ then $\{M_i\mid i\in P\}$ and $\{M_i\mid i\in Q\}$ generate distinct varieties: without loss of generality, there is $i\in P\backslash Q$ and then $\{M_i\mid i\in Q\}\models {\bf p}_i\approx {\bf q}_i$ but $\{M_i\mid i\in P\}\not\models{\bf p}_i\approx {\bf q}_i$.

Our  monoids $M_1,\dots$ will invoke a construction recently introduced by the first author in \cite[Part 3]{jck}. There a monoid $S_{I_4}^1$ is constructed from a specific instance of the computational problem \emph{positive 1-in-3SAT}. \ In the present article we will apply the same construction to 3-uniform hypergraphs, though now the development will be centred around the consideration of these monoids in the context of each other and with ${\bf A}_2^1$, with only some basic properties relating to ${\bf B}_2^1$ having close overlap with the results in \cite{jck}: Lemma \ref{lem:flexB21} below is essentially Lemma 7.1 of \cite{jck} and Lemma \ref{lem:B21} below is close to Lemma 7.2 of \cite{jck}, though neither directly implies the other.

For a $3$-uniform hypergraph $\mathcal{H}=(V,E)$ of girth at least 3, define a binary relation~$\equiv$ on $\binom{V}{2}$ as follows.  We write $\{u,v\}\equiv \{x,y\}$ if there is $w$ with $\{u,v,w\}$ and $\{x,y,w\}$ both hyperedges, or if neither $\{u,v\}$ nor $\{x,y\}$ extend to hyperedges.  By Lemma~\ref{lem:girth}(I), this relation is transitive, while symmetry and reflexivity are obvious.  Thus $\equiv$ is an equivalence relation.  We let $[u,v]$ and $[\{u,v\}]$ abbreviate $\{u,v\}/{\equiv}$.

Given any $3$-uniform hypergraph $\mathcal{H}=\langle V,E\rangle$ of girth at least $4$, we now construct three semigroups, $S_\mathcal{H}^{\dsharp}$, $S_\mathcal{H}^\sharp$, $S_\mathcal{H}$.  Each is a quotient of its predecessor so we describe $S_\mathcal{H}^{\dsharp}$ first.

The generators of $S_\mathcal{H}^{\dsharp}$ are $\{0,\st\}\mathrel{\dot\cup}V$, subject to the usual multiplicative properties of $0,1$ and the following rules
\begin{enumerate}
\item $\st^2=\st u\st=\st uv\st=0$ for all $u,v\in V$.
\item $uv=vu$ for each $u,v\in V$.
\item $uu=0$ for all $u\in V$.
\item $uv=0$ for all $u,v\in V$ for which no hyperedge of $E$ extends $\{u,v\}$.
\item $\st uvw\st=\st$ whenever $\{u,v,w\}$ is a hyperedge.
\end{enumerate}

The semigroup $S_\mathcal{H}^\sharp$ is obtained by adjoining two further kinds of relation to this presentation.
\begin{enumerate}
\item[(6)] $uvw=u'v'w'$ if $\{u,v,w\},\{u',v',w'\}\in E$.  Denote this element by $\se$.
\item[(7)] $\se\st\se=\se$.
\end{enumerate}
Finally, $S_\mathcal{H}$ is obtained by adding one further family of rules:
\begin{enumerate}
\item[(8)] $uv=u'v'$ if $\{u,v\}\equiv \{u', v'\}$.
\end{enumerate}

For any subset $s=\{u_1,\dots,u_i\}$ of $V$, we let $\overline{s}$ denote the product $u_1\dots u_i$, with $\overline{s}=1$ if $i=0$.

Finally, we let $M_\mathcal{H}^{\dsharp}$, $M_\mathcal{H}^\sharp$, $M_\mathcal{H}$ denote the monoids obtained from $S_\mathcal{H}^{\dsharp}$, $S_\mathcal{H}^\sharp$, $S_\mathcal{H}$ respectively, by adjoining an identity element.

The following lemma gathers some basic facts regarding products in $M_\mathcal{H}^{\dsharp}$, $M_\mathcal{H}^\sharp$ and $M_\mathcal{H}$.
\begin{lem}\label{lem:normalform1}
Let $\mathcal{H}$ be a $3$-uniform hypergraph of girth at least $4$.  The following hold in each of $M_\mathcal{H}^{\dsharp}$, $M_\mathcal{H}^{\sharp}$ and $M_\mathcal{H}$.
\begin{itemize}
\item A product of generators is nonzero if and only if after omitting occurrences of $1$ it is of the form
\[
\mathsf{s}_0\st\se_1\st\dots \se_{k-1}\st\mathsf{s}_{k}
\]
where $k\geq 0$,
\item  $\mathsf{s}_0=\overline{s_0}$ for some possibly empty subset $s_0$ of some hyperedge $e$ in $E$,
\item  $\mathsf{s}_k=\overline{s_k}$ for some possibly empty subset $s_k$ of some hyperedge $e$ in $E$,
\item  $\se_i=\overline{e_i}$ for some hyperedge $e_i$ of $E$.
\end{itemize}
Moreover, the product $\mathsf{s}_0\st\se_1\st \dots \se_{k-1}\st\mathsf{s}_{k}$
reduces to $\mathsf{s}_0\st\mathsf{s}_k$.
\end{lem}
\begin{proof}
Rule (1) ensures that a nonzero product must separate any occurrences of $\st$ with at least $3$ vertex elements.  However, condition (III) of Lemma \ref{lem:girth} ensures that any sequence of vertices $v_1,v_2,v_3,v_4$ must include either a repeat, or a pair of vertices that do not lie in any single hyperedge.  By Rule (2), we may rearranged the product $v_1v_2v_3v_4$ so that Rule (3) or (4) gives the value $0$.  Thus successive occurrences of $\st$ in a nonzero product are separated by a product of exactly $3$ vertices, and moreover, by Rules (2,3,4), all pairs of the three vertices must extend to a hyperedge, so that Lemma \ref{lem:girth}(II) shows that the three vertices must form a hyperedge in $E$.  This shows that such a product is of the form $\mathsf{s}_0\st\se_1\st\dots \se_{k-1}\st\mathsf{s}_{k}$ as described in the lemma.  Rule (5) enables the reduction to $\mathsf{s}_0\st\mathsf{s}_k$.  It remains to show that every product $\mathsf{s}_0\st\se_1\st\dots \se_{k-1}\st\mathsf{s}_{k}$ is nonzero.  This is basically trivial for $M_\mathcal{H}^{\dsharp}$, $M_\mathcal{H}^{\sharp}$, but for $M_\mathcal{H}$ we need to ensure that the extra law (8) does not enable the deduction of new $0$ products.  This rule only involves vertices, so it suffices to consider a product $uvw\neq 0$ in $M_\mathcal{H}^{\sharp}$ (that is, $\{u,v,w\}\in E$) and show that it is nonzero also after applications of rule (8).  But observe that if $[u,v]=[x,y]$, then $\{x,y,w\}\in E$ so that $uvw=xyw=\se$ in $M_\mathcal{H}$.
\end{proof}
 Note that if $k=0$ and $s_0$ is an actual hyperedge of $\mathcal{H}$, then $\overline{s_0}\st\overline{s_k}$ would be simply $\se$, while if $k=1$ but $s_1=\varnothing$, then $\overline{s_0}\st\overline{s_k}$ is just $\overline{s_0}\st$ (which would be $\se\st$ if $s_0\in E$ and the lemma is being applied to $M_\mathcal{H}^{\sharp}$ or to $M_\mathcal{H}$).

The following lemma is obvious and we omit the proof.
 \begin{lem}\label{lem:normalform2a}
 If $u,v,u',v'$ are subsets of hyperedges of $E$, then
 $\overline{u}\st\overline{v}=\overline{u'}\st\overline{v'}$ in $M_\mathcal{H}^{\dsharp}$ if and only if $u=u'$ and $v=v'$.
 \end{lem}

 The next lemma simply records the effect of adjoining rules (6,7,8).
 \begin{lem}\label{lem:normalform2}
 If $u,v,u',v'$ are subsets of hyperedges of $E$, then
 $\overline{u}\st\overline{v}=\overline{u'}\st\overline{v'}$ in $M_\mathcal{H}$ if and only if $|u|=|u'|$ and $|v|=|v'|$ and
 \begin{align*}
 u=u'&\text{ if $|u|=|u'|=1$}\\
 v=v'&\text{ if $|v|=|v'|=1$}\\
 [u]=[u']&\text{ if $|u|=|u'|=2$}\\
 [v]=[v']&\text{ if  $|v|=|v'|=2$}.
 \end{align*}
 \end{lem}
 \begin{lem}\label{lem:idempotent}
 If $u,v$ are subsets of hyperedges of $E$, then
 $\overline{u}\st\overline{v}$ is idempotent in $M_\mathcal{H}$ if and only if $u\cap v=\varnothing$ and $u\cup v\in E$.  The only other idempotent elements are $0$ and $1$.
 \end{lem}
 \begin{proof}
 This follows almost immediately from Lemma \ref{lem:normalform1} and Lemma \ref{lem:normalform2}.
 \end{proof}
 The following  lemma follows immediately from the observations so far and show that $\mathbb{B}_2^1\subseteq \mathbb{V}_m(M_\mathcal{H})$ for any hypergraph $\mathcal{H}$.
 \begin{lem}\label{lem:B21in}
 ${\bf B}_2^1$ is isomorphic to the submonoid of $M_\mathcal{H}$ generated as a monoid by $\{\st,\se\st\se\}$.
 \end{lem}
 The next lemma shows that $\mathbb{V}_m(M_\mathcal{H})\subseteq [\![x^2y^2\approx y^2x^2]\!]$.
 \begin{lem}\label{lem:commute}
 If $e,f$ are idempotents in $M_\mathcal{H}$, then $ef=fe$.  In particular, one of the following is true\up: $e=f$\up; $0,1\in \{e,f\}$\up; $ef=fe=0$.
 \end{lem}
 \begin{proof}
 Assume $e,f$ are idempotents, not including $0$ or $1$.  Let $u,v,u',v'$ be subsets of hyperedges with $e=\overline{u}\st\overline{v}$ and $f=\overline{u}'\st\overline{v}'$.  Assume $ef\neq 0$.  Now, Lemma \ref{lem:normalform1} implies that $v\cap u'=\varnothing$ and $v\cup u'\in E$ is a hyperedge.  By Lemma \ref{lem:idempotent} we also have $v\cap u=v'\cap u'=\varnothing$ and $u\cup v,u'\cup v'\in E$.  Hence $|u|=|u'|$ and $|v|=|v'|$ as well as $|u|+|v|=3$.  If $0\in \{|u|,|v|\}$ we are done as then $\overline{u}=\overline{u'}$ and $\overline{v}=\overline{v'}$.  Otherwise, there is no loss of generality in assuming that $|u|=|u'|=2$ and $|v|=|v'|=1$.  Then the condition $u\cup v\in E$ and $u'\cup v\in E$ implies that $u\equiv u'$ so that $\overline{u}=\overline{u'}$.  Then Lemma \ref{lem:girth}(I) implies that $v=v'$ also, as required.
 \end{proof}
\begin{lem}\label{lem:A21}
For any $3$-uniform hypergraph $\mathcal{H}=(V,E)$ of girth at least $4$, the monoid $M_\mathcal{H}$ lies in~$\mathbb{A}_2^1$.
\end{lem}
\begin{proof}
We wish to use the set $\binom{V}{2}$ of all pairs of vertices, but also a second marked copy of some of these.  Let $P$ denote a marked copy of the set of all pairs of vertices that do not extend to a hyperedge. So $\{u,v\}$ is a pair in $\binom{V}{2}$, while $\{u,v\}'$ is the copy of $\{u,v\}$ in $P$ (provided $\{u,v\}$ extends to a hyperedge in $E$).  This is just a notational device, and we will still consider $u$ and $v$ to be elements of the ``marked set'' $\{u,v\}'$ (if $\{u,v\}$ extends to a hyperedge) and abuse notation by writing $u,v\in\{u,v\}'$.  Let $N$ denote the disjoint union $\binom{V}{2}\mathrel{\dot\cup} P\mathrel{\dot\cup} V$.  We show that $M_\mathcal{H}^{\dsharp}$ is a quotient of $({\bf A}_2^1)^{N}$.
We allow a slight duplication notation and use $1$ to denote the constant tuple $(1,1,\dots)$.  Next, let $\hat\st\in ({\bf A}_2^1)^{N}$ be equal to $a$ on each $\{u,v\}\in\binom{V}{2}$ and $b$ on each $i\in P\cup V$.
Now let $\hat{v}$ be defined at $p\in \binom{V}{2}\cup P\cup V$ by
\[
\hat{v}(p):=\begin{cases}
1&\text{ if }p\in\binom{V}{2}\text{ and }v\in p\\
b&\text{ if }p\in\binom{V}{2}\text{ and }v\notin p\\
a&\text{ if }p\in P\text{ and }v\in p\\
1&\text{ if }p\in P\text{ and }v\notin p\\
1&\text{ if }p\in V\text{ and }v\neq p\\
a&\text{ if }v=p.
\end{cases}
\]
Let $T_\mathcal{H}$ be the submonoid of $({\bf A}_2^1)^N$ generated by $1$, $\hat{\st}$ and $\hat{u}$ for each $u\in V$.
Let $I$ be the ideal of $T_\mathcal{H}$ consisting of all elements that have a coordinate equal to $0$.

We are first going to verify that the ``hat'' versions of the laws~(1)--(5) in the definition of $M_\mathcal{H}^{\dsharp}$ hold in $T_\mathcal{H}/I$.
Observe that $\hat{u}\hat{v}=\hat{v}\hat{u}$ (verifying property~(2)), and we use this freely below.
Now $\hat{u}\hat{u}(u)=aa=0$, so that $\hat{u}\hat{u}\in I$ (verifying property~(3)).  Also, if $\{u,v\}\in\binom{V}{2}$ does not extend to a hyperedge, then $\{u,v\}'\in P$, so that $\hat{u}\hat{v}(\{u,v\}')=aa=0$, showing that $\hat{u}\hat{v}\in I$ (verifying property~(4)).  Thus by Lemma \ref{lem:girth}(IV), a product $\hat{u_1}\hat{u_2}\dots\hat{u_k}$ (for $u_1,\dots,u_k\in V$) can only lie outside of $I$ if:  $k=1$; or $k=2$ and $\{u_1,u_2\}$ is a $2$-element subset of a hyperedge; or $k=3$ and $\{u_1,u_2,u_3\}$ is a hyperedge in $E$.  In these three cases where $\hat{u_1}\hat{u_2}\dots\hat{u_k}\notin I$, we have $\hat{u_1}\dots\hat{u_k}(p)\in \{1,a\}$ for all $p\in P$.  Then $\hat{\st}\hat{u_1}\dots\hat{u_k}\hat{\st}(p)=b$ for $p\in P$.

Next observe that for any two element subset $\{u,v\}\in\binom{V}{2}$ we have $\hat{u}\hat{v}(\{u,v\})=1$.  Then $\hat{\st}\hat{u}\hat{v}\hat{\st}(\{u,v\})=a11a=aa=0$, showing that $\hat{\st}\hat{u}\hat{v}\hat{\st}\in I$.  Similarly, $\hat{\st}\hat{u}\hat{\st}$ and $\hat{\st}\hat{\st}$ lie in $I$ (verifying property~(1)).
On the other hand, if $u,v,w\in V$ are distinct, then $\hat{u}\hat{v}\hat{w}(p)=b$ for $p\in\binom{V}{2}$, and then $\hat{\st}\hat{u}\hat{v}\hat{w}\hat{\st}(p)=aba=a$.  As we have already established that $\hat{\st}\hat{u}\hat{v}\hat{w}\hat{\st}(p)=b$ when $p\in P$ and as $\hat{\st}\hat{u}\hat{v}\hat{w}\hat{\st}(p)=bab=b$ when $p\in V$ we have that $\hat{\st}\hat{u}\hat{v}\hat{w}\hat{\st}=\hat{\st}$ (verifying property (5)).

All of this shows that $T_\mathcal{H}/I$ satisfies all of the rules (1)--(5), and moreover, that the only products of generators (ignoring~$1$) not falling into $I$ are those of the form described in Lemma \ref{lem:normalform1} (but with added hat).  However, it is also straightforward to verify from Lemma \ref{lem:normalform2a} that no other equalities hold outside of those that hold in $M_\mathcal{H}^{\dsharp}$.  We briefly consider cases.  Products not involving $\hat{\st}$ are distinguished from those involving $\hat{\st}$ by taking the value $1$ at multiple points on $V$ (whereas those involving~$\hat{\st}$ take values in $\{a,b\}$ everywhere).  We can distinguish nonzero products of vertex elements $\hat{u_1}\dots\hat{u_k}$ and $\hat{v_1}\dots\hat{v_\ell}$ (for $k,\ell\leq 3$) by selecting a vertex $w$ lying in one of $\{u_1,u_2,u_3\}$ and $\{v_1,v_2,v_3\}$ and not the other: then  $\hat{u_1}\dots\hat{u_k}(w)\neq\hat{v_1}\dots\hat{v_\ell}(w)$.  The same idea distinguishes products of the form $\hat{u_1}\dots\hat{u_k}\hat{\st}x$ from those of the form $\hat{v_1}\dots\hat{v_k}\hat{\st}y$, as one will take a value in $\{a,ab\}$ at the point $w$ while the other will take a value in $\{b,ba\}$; a symmetric argument shows that we may separate $x\hat{\st}\hat{u_1}\dots\hat{u_k}$ from those of the form $y\hat{\st}\hat{v_1}\dots\hat{v_k}$.
Thus $T_\mathcal{H}/I\cong M_\mathcal{H}^{\dsharp}$.  Because $M_\mathcal{H}$ is a quotient of~$M_\mathcal{H}^{\dsharp}$ the lemma is proved.
\end{proof}

For $\ell>1$, an \emph{$\ell$-colouring} of a $3$-uniform hypergraph $\mathcal{H}=(V,E)$ is any function $\gamma:V\to \{0,\dots,\ell-1\}$ that leaves no hyperedge of $\mathcal{H}$ monochromatic; that is, $|\gamma(e)|\geq 2$ for each $e\in E$.  Equivalently, an $\ell$-colouring is a homomorphism from $\mathcal{H}$ to the  hypergraph on $\{0,\dots,\ell-1\}$ whose hyperedges consist of all subsets of size $2$ or $3$ (that is, it is not uniform).  The following deep result will be used extensively later, but here shows simply that there do exist hypergraphs of girth at least $4$ and which have high chromatic number; alternatively use Feder and Vardi \cite[Theorem~5]{fedvar}, noting Ham and Jackson \cite[Lemma~2.5]{hamjac}.
\begin{thm}[Erd\H{os} and Hajnal \cite{erdhaj}]\label{thm:EH}
For any $k,\ell>1$, there is a finite $3$-uniform hypergraph $\mathcal{H}$ with chromatic number $k$ and girth $\ell$.
\end{thm}

A $2$-colouring $\gamma$ of a hypergraph $\mathcal{H}=(V,E)$ is a \emph{majority two-colouring} if each hyperedge of $E$ receives two values of $1$ and one of $0$.  Equivalently, a majority $2$-colouring is a homomorphism into the structure on $\{0,1\}$ with ternary relation $\{(1,1,0),(1,0,1),(0,1,1)\}$, where we treat the set of hyperedges as a ternary relation $\{(u,v,w)\mid \{u,v,w\}\in E_\mathcal{H}\}$.  This is of course just the template for the classic \texttt{NP}-complete problem \emph{positive 2-in-3SAT}, though we will not need this fact in the present article; see \cite{jck} for development in that direction.  Because a majority $2$-colouring is a $2$-colouring, it follows from Theorem \ref{thm:EH} that there exist $3$-uniform hypergraphs of girth at least $4$ and that are not majority $2$-colourable.  The following lemma is essentially a special case of \cite[Lemma~7.2]{jck}, though the assumptions here are perhaps slightly different.
\begin{lem}\label{lem:B21}
If $M_\mathcal{H}\in\mathbb{B}_2^1$, then $\mathcal{H}$ has a majority $2$-colouring.
\end{lem}
\begin{proof}
Assume there is $n\in \mathbb{N}$ and ${\bf M}\leq ({\bf B}_2^1)^n$ is such that there exists a surjective homomorphism $\eta$ from ${\bf M}$ to $M_\mathcal{H}$.  Let $\hat{\st}$ denote a $\mathscr{J}$-minimal member of $\eta^{-1}(\st)$.  For each $u\in V$, let $\hat{u}$ be a fixed element of $\eta^{-1}(u)$.  Also, for any hyperedge $\{u,v,w\}$ we have $\hat{\st}\hat{u}\hat{v}\hat{w}\hat{\st}\leq_\mathscr{J}\hat{\st}$.  As $\hat{\st}\hat{u}\hat{v}\hat{w}\hat{\st}\in\eta^{-1}(\st u v w\st)=\eta^{-1}(\st)$, the $\mathscr{J}$-minimality of $\hat{\st}$ ensures that $\hat{\st}\hat{u}\hat{v}\hat{w}\hat{\st}=\hat{\st}$.  (Note that in general we may only have $\mathscr{H}$-relatedness rather than equality, but this semigroup has trivial $\mathscr{H}$ relation.)

Now $\st^2=0$ and so it follows that there is $i\leq n$ such that $\hat{\st}(i)\in\{a,b\}$.  Up to symmetry we assume without loss of generality that $\hat{\st}(i)=b$.  Now for each $u\in V$ we have $\hat{u}(i)\in\{1,a\}$, because $\hat{\st}\hat{u}$ and $\hat{u}\hat{\st}$ are in the same $\mathscr{J}$-class as $\hat{\st}$.  Define $\nu:V\to\{0,1\}$ by $\nu(u)=1$ if $\hat{u}(i)=1$ and $\nu(u)=0$ if $\hat{u}(i)=a$.  Now for each hyperedge $\{u,v,w\}$, as $(\hat{\st}(i))^2=b^2=0$ but $\hat{\st}(i)\hat{u}(i)\hat{v}(i)\hat{w}(i)\hat{\st}(i)=\hat{\st}(i)=b$, it follows that exactly one vertex $x\in \{u,v,w\}$ has $\hat{x}(i)=a$, with the remaining two vertices in $\{u,v,w\}\backslash\{x\}$ having value $1$ at $i$.  In other words, for each hyperedge, exactly one vertex has $\nu$-value $0$ and exactly two have $\nu$-value $1$.  In other words, $\mathcal{H}$ has a majority $2$-colouring.
\end{proof}
The following lemma comes from \cite[Lemma 7.1]{jck}, but we give a sketch proof here for completeness.
\begin{lem}\label{lem:flexB21}
Assume that $\mathcal{H}$ is a 3-hypergraph with the following properties\up:
\begin{enumerate}
\item if $u\neq v$ then there are majority $2$-colourings $\gamma_1,\gamma_2,\gamma_3$ with $\{(\gamma_i(u),\gamma_i(v))\mid i=1,2,3\}=\{(1,1),(0,1),(1,0)\}$\up;
\item if $u\neq v$ and no hyperedge extends $\{u,v\}$ then there is a majority $2$-colouring $\gamma$ with $(\gamma(u),\gamma(v))=(0,0)$.
\end{enumerate}
Then $M_\mathcal{H}\in \mathbb{B}_2^1$.
\end{lem}
\begin{proof}
Let $P$ denote the set of all majority $3$-colourings of $\mathcal{H}$ and let $T$ denote the subsemigroup of $({\bf B}_2^1)^P$ generated by
$\hat{\st}(\gamma)=b$ and
\[
\hat{u}(\gamma)=\begin{cases}
\gamma(u)&\text{ if $\gamma(u)=1$}\\
a&\text{ if $\gamma(u)=0$}
\end{cases}
\] for each $\gamma\in P$ and $u\in V$.  Let $I$ be the ideal of $T$ generated by all elements $x$ of $T$ for which there exists $\gamma\in P$ with $x(\gamma)=0$.  Then $T/I$ is isomorphic to $M_\mathcal{H}$ under the map that sends $\hat{\st}$ to $\st$ and each $\hat{u}$ to $u$.  The details from here are completely routine and are omitted.
\end{proof}
The following lemma will be used later, and gives a broad class of $3$-uniform hypergraphs that are majority $2$-colourable in enough ways to satisfy both conditions in Lemma~\ref{lem:flexB21}.  Cycle-free $3$-uniform hypergraphs are also known as \emph{$3$-uniform hyperforests}.
\begin{lem}\label{lem:cyclefree}
Let $\mathcal{H}$ be a cycle-free $3$-uniform hypergraph.  Then $\mathcal{H}$ satisfies conditions \up(1\up) and \up(2\up) of Lemma \ref{lem:flexB21}, so that $M_\mathcal{H}\in\mathbb{B}_2^1$.
\end{lem}
\begin{proof}
The proof is by induction on the number of hyperedges.  The statement is clearly true if there are no hyperedges or if there is a single hyperedge.  Now assume that the statement is true when for all cycle-free $3$-uniform hypergraphs with at most $k$ hyperedges and consider the case where $\mathcal{H}=(V,E)$ has $k+1$ distinct hyperedges.  Note that being cycle free ensures that two hyperedges overlap in at most one vertex.  As $\mathcal{H}$ has only finitely many hyperedges ($k+1$ in fact) and no cycles, we can find a hyperedge $e=\{u,v,w\}\in E$ with at most one vertex from $u,v,w$ appearing in any hyperedge in $E\backslash\{e\}$.  First assume that $u$ lies in at least one hyperedge in $E\backslash\{e\}$; so, vertices $v$ and $w$ do not lie in any other hyperedge.  The induced sub-hypergraph on $V\backslash\{v,w\}$ satisfies the conditions of Lemma \ref{lem:flexB21}.  In particular there is a majority $2$-colouring $\gamma_1$ of this subgraph that gives $u$ the colour $0$ and a majority $2$-colouring $\gamma_2$ that gives $u$ the value $1$.  Then $\gamma_1$ can be extended to $v,w$ by giving them both the value $(1,1)$, and $\gamma_2$ may be extended to colour $(v,w)$ by either of $(0,1)$ and $(1,0)$.  Note that in its extended form, $\gamma_1$ also colours $(u,v)$ and $(u,w)$ by $(0,1)$, while the two described extensions of $\gamma_2$ colour $(u,v)$ and $(u,w)$ by both $\{(1,1),(1,0)\}$.  Thus each pair from $u,v,w$ satisfies the conditions of Lemma~\ref{lem:flexB21}.  It now suffices to show that for any vertex $u'\in V\backslash\{u,v,w\}$ that the pair $(u',v)$ can be coloured in all of the four possible ways: $(0,0),(0,1),(1,0),(1,1)$.  First observe that $(u',u)$ may be coloured $(0,1),(1,0),(1,1)$, and possibly $(0,0)$ (unless $\{u,u'\}$ extends to some hyperedge in $E\backslash\{E\}$).  If $\gamma$ is a majority $2$-colouring of $(V\backslash\{v,w\},E\backslash\{e\})$ giving  $(u',u)$ the pair of colours $(0,1)$ or $(1,1)$, then $\gamma$ can be extended to a majority $2$-colouring of $\mathcal{H}$ giving $v$ either value $0$ or $1$ (as $u$ is taking the value $1$).  This gives all combinations $(u',v)\in\{(0,0),(0,1),(1,0),(1,1)\}$, as required.  This completes the inductive step when $e$ shares a vertex in common with some other hyperedge in $E$.  The case where $e$ shares no vertices is very similar though easier and we omit the details.
\end{proof}

\section{Proof of Theorem \ref{thm:AB}}
In this section we prove Theorem \ref{thm:AB} by adapting an idea from \cite[\S3.5]{jckmck}: constructing an identity around any hypergraph $\mathcal{G}$ and whose satisfaction by $M_\mathcal{H}$ will relate to the homomorphism properties of $\mathcal{G}$ and $\mathcal{H}$.
For any 3-hypergraph $\mathcal{H}=(V,E)$ of girth at least $4$ we define a word $p_\mathcal{H}$ as follows.  First, the variables appearing in $p_\mathcal{H}$ will be a variable $y$ along with a variable $x_u$ for each $u\in V$.  Now let $\bw_1,\bw_2,\dots,\bw_n$ be a sequence of words $x_ux_vx_w$, satisfying the following conditions.
\begin{enumerate}
\item[($\alpha$)] $x_ux_vx_w$ appears in the list if and only if $\{u,v,w\}\in E$ (in particular, all permutations of $x_ux_vx_w$ appear in the list).  Note that $x_ux_vx_w$ may appear more than once in the list.
\item[($\beta$)] for every pair of vertices $u,v$ (including when $u=v$), there is an $i\leq n-1$ such that $\bw_i$ finishes with $x_u$ and $\bw_{i+1}$ starts with $x_v$.
\item[($\gamma$)] for every hyperedge $\{u,v,w\}$ and every vertex $u'\notin\{u,v,w\}$, there is $j$ such that $\bw_j$ is  $x_ux_vx_w$, and such that $\bw_{j-1}$ ends in $x_{u'}$ while $\bw_{j+1}$ begins with $x_{u'}$.
\end{enumerate}
Provided that each vertex of $\mathcal{H}$ is contained within a hyperedge (which we will assume throughout), it is clear that such lists exist: if $L$ is a list of all $6|E|$ possible triples $x_ux_vx_w$ for hyperedge $\{u,v,w\}$, then the list obtained by concatenating all possible permutations of $L$ has this property.
Now let $p_\mathcal{H}$ be the word $\prod_{1\leq i\leq n}(y\bw_i)^2y$.

\begin{lem}\label{lem:HnotpH}
For any hypergraph $\mathcal{G}$ we have $M_\mathcal{G}\not\models p_\mathcal{G}\approx p_\mathcal{G}^2$.
\end{lem}
\begin{proof}
Simply assign $y$ the value $\st$ and $x_u$ the value $u$; then the left hand side takes the value $\st$ while the right hand side takes the value $0$.
\end{proof}

In the case of the semigroup construction $S_\mathcal{H}$, it is possible to show that failure of the law $p_\mathcal{G}\approx p_\mathcal{G}^2$ precisely captures the existence of a homomorphism from $\mathcal{G}$ to $\mathcal{H}$ (at least assuming that $\mathcal{G},\mathcal{H}$ are not majority $2$-colourable and have girth at least $4$).  This claim is not used in the paper and we omit the proof, though we note it is very similar but easier than the proof Lemma \ref{lem:homommon} below: simply omit all cases where the element $1$ is considered.  In the monoid case we need much stronger assumptions on our hypergraphs to achieve a comparably useful result.

We will use the notation $\wp_{\leq i}(S)$ to denote the set of all subsets of a set $S$ with at most $i$ elements.  Let us say that two $3$-uniform hypergraphs $\mathcal{G}=(V_G,E_G)$ and $\mathcal{H}=(V_H,E_H)$ are \emph{wildly incomparable} if the following property holds:
\begin{itemize}
\item both $\mathcal{G}$ and $\mathcal{H}$ have girth at least $4$ and are not $5$-colourable; and either
\begin{itemize}
\item the girth of $\mathcal{G}$ is greater than $3|V_H|+1$ and its chromatic number is more than $|\wp_{\leq 3}(V_H)|+1$; or
\item the girth of $\mathcal{H}$ is greater than $3|V_G|+1$ and its chromatic number is more than $|\wp_{\leq 3}(V_G)|+1$.
\end{itemize}
\end{itemize}
Theorem \ref{thm:EH} can be used to construct an infinite family of pairwise wildly incomparable hypergraphs.  Starting from any $3$-uniform
hypergraph~$\mathcal{G}_0$ of girth at least~$4$ and with chromatic number more than~$5$, inductively construct $\mathcal{G}_{i+1}$ using Theorem~\ref{thm:EH} by setting $k:=|\wp_{\leq 3}(V_{G_i})|+2$ and $\ell=3|V_{G_i}|+2$.  Of course, the growth in size of these hypergraphs is prodigious, and the identities $p_{\mathcal{G}_i}\approx p_{\mathcal{G}_i}^2$ as constructed above are larger still.

\begin{lem}\label{lem:homommon}
Let $\mathcal{G}=(V_G,E_G)$ and $\mathcal{H}=(V_H,E_H)$  be two wildly incomparable 3-hypergraphs.  Then $M_\mathcal{H}\models p_\mathcal{G}\approx p_\mathcal{G}^2$.
\end{lem}
\begin{proof}
Let $\theta$ be an assignment from the variables $\{y\}\cup\{x_u\mid u\in V_\mathcal{G}\}$ into $M_\mathcal{H}$.  Our goal is to show that $\theta(p_\mathcal{G})$ is idempotent.  If $\theta(p_\mathcal{G})= 0$ then we are finished.    If $\theta(y)=1$ then $\theta(p_\mathcal{G})=\theta(\prod_{1\leq i\leq n}(\bw_i)^2)$ which is idempotent because the law $x^2y^2\approx (x^2y^2)^2$ is satisfied by $M_\mathcal{H}$.  Moreover we may assume that $\theta(y)$ is not idempotent because if $\theta(y)$ is an idempotent other than $1$ it is easy to see that $\theta(p_\mathcal{G})\in\{\theta(y),0\}$.

Thus for the remainder of the proof we may assume that $\theta(p_\mathcal{G})\neq 0$ (so that no variable is assigned $0$ by $\theta$) and that $\theta(y)$ is not idempotent.
We let $V_1$ denote those vertices $u$ with $\theta(x_u)=1$ and let $V_\natural$ denote those vertices that take a value equal to a product of vertex elements of $M_\mathcal{H}$.   The notation $V_{\st}$ will be used for the subset $V_G\backslash (V_1\cup V_\natural)$ consisting of those vertices $u$  with $\theta(x_u)\mathrel{\mathscr{J}}\bt$.  The vertices in $V_\st$ are the vertices for which $\theta(x_u)$ is of the form $\overline{s}_0 \bt\overline{s}_1$, where $s_0$ and $s_1$ are possibly empty subsets of hyperedges in $E_H$.  With each  $u\in V_\st$ we define the two numbers $\ell_u:=|s_0|$ and $r_u:=|s_1|$ in $\{0,1,2,3\}$.

Now we wish to show that $\theta(y)$ is not of the form $\overline{s}$ for some nonempty subset of a hyperedge.  Assume, for contradiction, that $\theta(y)$ is of the form $\overline{s}$ for some nonempty subset of a hyperedge.  Condition~(III) for $\mathcal{H}$ and Property (4) for $M_\mathcal{H}$, ensures that $V_\st$ is nonempty, while Condition ($\beta$) for the sequence $\bw_1,\bw_2,\dots$ shows that for every pair of variables $u,v\in V_G$ (possibly identical), the product $\theta(x_u)\theta(y)\theta(x_v)$ appears in $\theta(p_\mathcal{G})$.  Thus for every pair $u,v\in V_\st$ (not necessarily distinct) we have $r_u+|s|+\ell_v=3$, from which it follows that $r_u=r_v$ and $\ell_u=\ell_v$; for example $r_u+|s|+\ell_v=r_u+|s|+\ell_u$ yields $\ell_u=\ell_v$.   It now follows that no hyperedge $\{u,v,w\}$ of $\mathcal{G}$ can contain two vertices $u,v\in V_\st$ because then $x_ux_v$ is a subword of $p_\mathcal{G}$ and $\theta(x_u)\theta(x_v)\neq 0$ implies that $r_u+\ell_v=3$, contradicting $r_u+|s|+\ell_v=3$ and the fact that $|s|\geq 1$.  There is also no hyperedge $\{u,v,w\}$ of $\mathcal{G}$ containing both a vertex from $V_\st$ (say, $u$) and a vertex from $V_\natural$ (say, $v$), as the word $x_ux_vyx_u$ appears in $p_\mathcal{G}$ and gives $r_u+|s|+|\theta(v)|+\ell_v=3$, which again contradicts $r_u+|s|+\ell_v=3$.  Thus every hyperedge intersecting $V_\st$ consists of one vertex from $V_\st$ and two vertices from $V_1$.

Next we note that no hyperedge $\{u,v,w\}$ of $\mathcal{G}$ is a subset of $V_1$ either, as we may find $i$ and $u'\in V_\st$ such that $\bw_{i+1}=x_ux_vx_w$ while $\bw_i$ finishes with $x_{u'}$ and $\bw_{i+2}$ starts with $x_{u'}$.
Then $\theta(x_{u'}yx_ux_vx_wyx_{u'})\neq 0$ would require that $r_{u'}+2|s|+\ell_{u'}=3$, contradicting $r_{u'}+\ell_{u'}=3-|s|$ and $|s|\geq 1$.  Similarly, no hyperedge $\{u,v,w\}$ of~$\mathcal{G}$  can contain more than one vertex from $V_\natural$; if $u,v\in V_\natural$, then as we know that $w\notin V_\st$, we have that the product $\theta(y)\theta(x_u)\theta(x_v)\theta(x_w)\theta(y)$ is a product of vertex elements of length more than $3$.  But Condition~(III) for $\mathcal{H}$ and Property (4) for $M_\mathcal{H}$ show that that such products are $0$.  Thus every hyperedge consists of two vertices from $V_1$ and one vertex from $V_\st\cup V_\natural$.  But this is a (majority) $2$-colouring, which contradicts the choice of $\mathcal{G}$.   This completes the proof that $\theta(y)$ is not of the form~$\overline{s}$.

We have now shown that $\theta(y)$ is of the form $\overline{s}_0 \bt\overline{s}_1$, where $s_0$ and $s_1$ are possibly empty subsets of hyperedges in $E_H$.  We now split the proof into two streams.

Stream 1.  \emph{$\mathcal{G}$ has chromatic number greater than $|\wp_{\leq 3}(V_H)|+1$ and girth greater than $3|V_H|+1$}.\\
We are going to argue that the assumptions on $\theta(p_\mathcal{G})\neq 0$ and $\theta(y)\neq \theta(y)^2$ lead to the contradictory conclusion that $\mathcal{G}$ can be coloured by fewer than $|\wp_{\leq 3}(V_H)|+1$ colours.  We refine the set $V_\natural$ as follows.  Each vertex $v\in V_\natural$ has $\theta(x_v)$ equal to a nonempty subset of $E_H$, of which there are at most $|\wp_{\leq 3}(V_H)|-1$ elements.  For each such nonempty subset $s\subseteq e\in E_H$, let $V_s$ denote $\{v\in V_G\mid \theta(x_v)=\overline{s}\}$.  The family $\{V_s\mid s\subseteq e\in E_H \}\cup\{V_\st,V_1\}$ partitions $V_G$ into at most $|\wp_{\leq 3}(V_H)|+1$ elements.  We will show that no hyperedge of $\mathcal{G}$ is coloured uniformly by this partition, contradicting the chromatic number assumptions on $\mathcal{G}$.

Let $\{u,v,w\}$ be a hyperedge of $\mathcal{G}$ with $u\in V_\st$.  We show by contradiction that $v\not\in V_\st$ (in fact we can show that $v\in V_1$, but we will need only the weaker conclusion here).  Assume then, that $v\in V_\st$.  Now, each of the words $x_ux_vy$, $x_uyx_u$, $x_vx_uy$ are subwords of $p_\mathcal{G}$ from which we easily obtain  $r_u=r_v=r_y=:r$ and $\ell_u=\ell_v=\ell_y=:\ell$ and
$\ell+r=3$.
As $\theta(y)$ is not idempotent, it follows that $\{\ell,r\}=\{1,2\}$; say, $y=w\bt \overline{s}$, where $w$ is a vertex and $s$ is a two-element subset of some hyperedge.  We also let $u', s_u'$ be such that $\theta(x_u)=u'\bt \overline{s_u}$ and $v', s_v$  be such that $\theta(x_u)=v'\bt \overline{s_v}$.  Because $\theta(y)$ is not idempotent, we have that $\{w\}\cup s$ is \emph{not} a hyperedge.  But as $x_uy$ and $x_vy$ appear in $p_\mathcal{G}$ we must have $s_u\cup\{w\}$ and $s_v\cup \{w\}$ are hyperedges.  But as $x_ux_v$ appears in $p_\mathcal{G}$ we have that $s_u\cup\{v'\}$ is a hyperedge.  Condition (I) for $\mathcal{G}$ then shows that $v'=w$ and by symmetry, that $u=v'=w$.  But this contradicts the fact that $\theta(y)$ is not idempotent.  Hence  $v\not\in V_\st$ as claimed.

Now consider a hyperedge $\{u,v,w\}$ that does not intersect with $V_\st$.  We cannot have $\{u,v,w\}\subseteq V_1$, as this would mean that $\theta(yx_ux_vx_wy)=\theta(y)\theta(y)=0$, allowing for the fact that $\theta(y)$ is not idempotent. So there is at least one of $u,v,w\in V_\natural$.  The argument will be complete if we can show that if $u,v\in V_\natural$ then $\theta(x_u)\neq \theta(x_v)$, as then $u,v$ lie in a different blocks of the family $\{V_1\}\cup\{V_s\mid s\subseteq e\in E_H\}$.  But this follows immediately from Properties (2) and (3) of $M_\mathcal{H}$.  This concludes the proof in the case that $\mathcal{G}$ is not $(|\wp_{\leq 3}(V_H)|+1)$-colourable.

Stream 2.  \emph{The girth of $\mathcal{H}$ is larger than $3|V_G|+1$.}
\\
We are going to argue that $\mathcal{G}$ can be $5$-coloured, a contradiction.  The partition of $V_G$ will be as follows: $\{V_\st,V_1,V_a,V_b,V_c\}$ where $V_a$, $V_b$ and $V_c$ partition $V_\natural$ in a way to be described shortly.
We know already that if $u\in V_\st$ and $\{u,v,w\}\in E_G$, then $v,w\in V_1$, so that every edge intersecting $V_\st$ is coloured in a valid way by the proposed partition.  Thus we only need to select $V_a,V_b,V_c$ in such a way that does not uniformly colour hyperedges $\{u,v,w\}\subseteq V_1\cup V_\natural$.  Let $V_{\natural,H}$ denote those vertices $v'$ in $V_H$ for which there exists $v\in V_\natural$ with $\theta(x_v)=v'$.  There are at most $|V_G|$ such vertices, so the sub-hypergraph of $\mathcal{H}$ on $V_{\natural,H}$ is a hyperforest.  By Lemma \ref{lem:cyclefree} we may  $2$-colour this hyperforest.  Define $V_a\subseteq V_\natural$ consist of those vertices $v$ for which $\theta(v)\in V_{\natural,H}$ and that are coloured by $0$. Define $V_b\subseteq V_\natural$ to consist of those vertices~$v$ for which $\theta(v)\in V_{\natural,H}$ and that are coloured by $1$.  Let $V_c:=V_\natural\backslash(V_a\cup V_b)$, which consists of all vertices $v$ in $V_\natural$ for which $\theta(x_v)$ is a proper product of vertex elements of $M_\mathcal{H}$.  We now show that our partition is a $5$-colouring of $\mathcal{G}$, which is the desired contradiction.

We know that no edge is uniformly coloured by $V_1$.  Thus any hyperedge intersecting $V_1$ is validly coloured (as at least one vertex lies in $V_\natural$, and at least one in $V_1$).  Now consider a hyperedge $\{u,v,w\}\subseteq V_\natural$.  Then by condition~(III), property~(4) and the fact that $\theta(x_u)\theta(x_v)\theta(x_w)\neq 0$ imply that $\theta(x_u),\theta(x_v),\theta(x_w)\in V_{\natural,H}$ and $\{\theta(x_u),\theta(x_v),\theta(x_w)\}\in E_H$.  But then the $2$-colouring of such edges ensures that both $V_a$ and $V_b$ intersect $\{u,v,w\}$, as required.  This concludes the proof in the case that the girth of $\mathcal{H}$ is larger than $3|V_G|+1$, the final case for the proof.
\end{proof}

\begin{thm}
The interval $[\mathbb{B}_2^1,[\![ x^2y^2\approx y^2x^2]\!]\wedge \mathbb{A}_2^1]$ has cardinality $2^{\aleph_0}$.
\end{thm}
\begin{proof}
For each $n>2$, let $H_{n,k}$ denote a finite $3$-uniform hypergraph with no cycles of length at most $n$ and which is not $k$-colourable.  Now we inductively define the following sequence of hypergraphs.  Let $G_1:=H_{4,5}$ and for $i\geq 1$ define
\[
G_{i+1}=H_{3|G_i|+2,|\wp_{\leq 3}(V_{G_i})|+2}.
\]
We observe that for $i\neq j$, the graphs $G_i$ and $G_j$ are wildly incomparable.
For each distinct pair of subsets $P,Q\subseteq \mathbb{N}$, we show that $\mathbb{V}_m(\{M_{G_i}\mid i\in P\})\neq \mathbb{V}_m(\{M_{G_i}\mid i\in Q\})$.  Without loss of generality, assume that $k\in P\backslash Q$.  Then $\{M_{G_i}\mid i\in Q\}\models p_{G_k}\approx p_{G_k}^2$ by Lemma \ref{lem:homommon}.  But $\{M_{G_i}\mid i\in P\}\not\models p_{G_k}\approx p_{G_k}^2$ by Lemma \ref{lem:HnotpH}.

By Lemma \ref{lem:commute}, each variety $\mathbb{V}(\{M_{G_i}\mid i\in P\})$ satisfies $x^2y^2\approx y^2x^2$ and by Lemma~\ref{lem:B21} and the fact that none of the $G_i$ is majority $2$-colourable, all properly contain $\mathbb{B}_2^1$.
\end{proof}
\begin{thm}
No finite set of identities defines $\mathbb{B}_2^1$ within $\mathbb{A}_2^1$.
\end{thm}
\begin{proof}
Fix any $n\in\mathbb{N}$, and consider a $3$-uniform hypergraph $\mathcal{H}=(V,E)$ with no cycles of length at most $m:=3\binom{6n}{2}$, and chromatic number at least $3$.  In particular then, $M_\mathcal{H}\notin \mathbb{B}_2^1$ by Lemma \ref{lem:B21} because it is not $2$-colourable, let alone majority $2$-colourable; but $M_\mathcal{H}\in \mathbb{A}_2^1$ by Lemma \ref{lem:A21}.  Let $S\subseteq M_\mathcal{H}$ be a subset of size at most~$n$.  We show that the submonoid generated by $S$ lies in $\mathbb{B}_2^1$. This will prove the theorem, as for any finite set of identities~$\Sigma$ that are satisfied by ${\bf B}_2^1$ we may choose~$n\in \mathbb{N}$ to be the maximum number of variables appearing in $\Sigma$.  Any evaluation $\theta$ of an identity $\bu\approx \bv$ from $\Sigma$ into $M_\mathcal{H}$ maps into a submonoid generated by at most $n$-elements.  If we have shown that such a submonoid lies in $\mathbb{B}_2^1$, then as ${\bf B}_2^1$ satisfies $\bu\approx \bv$ we must have $\bu\theta=\bv\theta$.  This will show that $M_\mathcal{H}$ satisfies $\Sigma$, showing that $\Sigma$ is insufficient to define $\mathbb{B}_2^1$ within $\mathbb{A}_2^1$.

We begin by selecting a subset $V_1$ of the vertices $V$ to represent the elements of~$S$.  By Lemma \ref{lem:normalform1}, each element $s\in S$ can be written as $\overline{s_0}$ for some subset $s_0$ of size at most $3$ of $V$, or as $\overline{s_0}\bt\overline{s_1}$ for subsets $s_0,s_1\subseteq V$ of size at most $3$.  Moreover the subsets $s_0$ and $s_1$ (if applicable) extend to hyperedges of $\mathcal{H}$.  Thus each $s$ can be associated with either a hyperedge $\{u_0,v_0,w_0\}$, or a pair of hyperedges $\{u_0,v_0,w_0\}$, $\{u_1,v_1,w_1\}$ with $s_0$ equal to one of $\varnothing$, $\{u_0\}$, $\{u_0,v_0\}$ or $\{u_0,v_0,w_0\}$ and $s_1$ (if applicable) equal to one of $\varnothing$, $\{u_1\}$, $\{u_1,v_1\}$ or $\{u_1,v_1,w_1\}$.  Note that the choice of $\{u_0,v_0,w_0\}$ may not be unique: for example if $s_0=\varnothing$, or if $s_0=\{u_0\}$, or even if $s_0=\{u_0,v_0\}$ but $\{u_0,v_0\}\equiv\{u_0',v_0'\}$ (and similarly for $s_1$).  Nevertheless, for each $s\in S$ we fix a choice for the one or two hyperedges, and let $V_1$ denote the union over these choices for all $s\in S$. Note that $|V_1|\leq 6n<m$ so that the induced subgraph $\mathcal{H}_1$ of $\mathcal{H}$ on the vertices $V_1$ is a hyperforest due to the condition on the girth of $\mathcal{H}$.

The set $\{\bt\}\cup V_1$ generates a submonoid $M'$ of $M_\mathcal{H}$ that contains the submonoid generated by $S$.  At this point we can use Lemma \ref{lem:cyclefree} to deduce that the monoid~$M_{\mathcal{H}_1}$ lies in $\mathbb{B}_2^1$, however this is not sufficient to prove that $M'\in \mathbb{B}_2^1$ because the monoid~$M_{\mathcal{H}_1}$ is not typically identical to $M'$, even if it is similar.  Both have the same generators and many of the rules determining $M_{\mathcal{H}_1}$ hold also on $M'$, but there are differences due to the fact that some products in $M'$ have properties determined in $M_\mathcal{H}$ that are not witnessed in~$M_{\mathcal{H}_1}$.  The first difference arises from the property~(4) in the definition of the $M_\mathcal{G}$ construction.  If $\{u,v\}$ is a set of vertices in $\mathbb{H}_1$ that extends to a hyperedge in $\mathbb{H}$ but not $\mathbb{H}_1$, then $M_\mathcal{H}$ satisfies $uv=0$ while~$M'$ does not.  A second difference arises from property~(8) as it is possible that $\{u,u'\}$ and $\{v,v'\}$ may be sets of vertices in $\mathcal{H}_1$ with $\{u,u'\}\equiv \{v,v'\}$ in $\mathcal{H}$ but not in $\mathcal{H}_1$: then~$M'$ satisfies $uv=u'v'\neq 0$ while in $M_{\mathcal{H}}$ they will be $0$.

In order to eliminate this problem we utilise the full value of $m$.  Let
\[
V_2:=\bigcup\{e\in E\mid e\text{ extends  a $2$-element subset of $V_1$}\},
\]
and let $\mathcal{H}_2$ be the induced subgraph.  As before, it is not necessarily the case that~$M_{\mathcal{H}_2}$ is a submonoid of $M_\mathcal{H}$, however we will show that $M'$ is a submonoid of $M_{\mathcal{H}_2}$ and that $M_{\mathcal{H}_2}\in \mathbb{B}_2^1$, which is all we need (as then $M'\in \mathbb{B}_2^1$ as required).  Now, every vertex $v$ of $V_1$ lies in some hyperedge of $\mathbb{H}$, so $V_1\subseteq V_2$ by the definition of $V_2$.  Moreover, every pair $\{u,v\}$ that extends to a hyperedge in $\mathcal{H}$ extends to a hyperedge in $\mathbb{H}_2$, again by the definition of $V_2$.  Thus for products involving only $\bt$ and vertices in $V_1$, all of the conditions (1)--(8) that hold in~$M_\mathcal{H}$ also hold in~$M_{\mathcal{H}_2}$.  So $M'$ (which is generated by $V_1$ as a submonoid of $M_\mathcal{H}$) is a submonoid of~$M_{\mathcal{H}_2}$.  Now we observe that $|V_2|\leq m$ because for each $2$-element subset $\{u,v\}$ of $V_1$, there is at most one hyperedge that extends $\{u,v\}$.  For thus for each of the at most~$\binom{6n}{2}$ pairs of vertices in $V_1$ we have included up to $3$ vertices in $V_2$, showing that $|V_2|\leq 3\binom{6n}{2}=m$.  As $\mathcal{H}$ contains no cycles of length at most $m$, it follows that $\mathcal{H}_2$ is a hyperforest, so that $M_{\mathcal{H}_2}\in \mathbb{B}_2^1$ by Lemma \ref{lem:cyclefree}.
\end{proof}

\section{From $\mathbb{B}_2^1$ to $\mathbb{M}(\bz_\infty)$ via $x^2y\approx yx^2$.}
We now turn our attention to the interval between $\mathbb{M}(\bz_\infty)$ and $\mathbb{B}_2^1$.  We introduce some standard notation in the analysis of semigroup words.  We fix some ambient countably infinite set $\mathcal{X}$ of variables from which the letters in our words will be drawn.  The \emph{content} of a word~$\bw$, denoted $\con(\bw)$ is the set of letters that have at least one occurrence in $\bw$.  Letters with as single occurrence are said to be \emph{simple} in $\bw$, and the set of simple letters for $\bw$ is denoted $\simp(\bw)$.  The set $\con(\bw)\backslash\simp(\bw)$ consists of the nonsimple letters of $\bw$ that occur in $\bw$, and is denoted by $\non(\bw)$.  The first letter, or head, to appear in~$\bw$ (from left to right) is denoted~$\h(\bw)$, while the dual notion of being the first appear from right to left (or the last from left to right) is the tail, denoted~$\tail(\bw)$.  Often it will be useful to delete letters from a word and consider the string that remains: if $A\subseteq\mathcal{X}$ and $\bw$ is a word, then $\bw_A$ will denote the word obtained by deleting all letters from $\bw$ that do not appear in~$A$.

\begin{lem} \label{lem: B21 xyxy}
Let $\bu \approx \bv$ be any identity satisfied by the
monoid ${\bf B}_2^1$ and $\bu =(xy)^k$ for some $k \geq 2$. Then
$\bv=(xy)^h$ for some $h \geq 2$.
\end{lem}

\begin{proof}
Suppose that $\mathsf{con} (\bv) \ne \mathsf{con} (\bu)$, say $x \notin \mathsf{con} (\bv)$.
Then letting $\varphi: \mathcal{X} \rightarrow {\bf B}_2^1$ be the substitution that maps $x$ to $a$ and $y$ to $1$, we obtain the contradiction $\varphi(\bu) = a^k=0 \ne 1=\varphi(\bv)$. Thus $\mathsf{con} (\bv) = \mathsf{con} (\bu)=\{x, y\}$.

Suppose that $\mathsf{non} (\bv) \ne \mathsf{non} (\bu)$, say $x \in \mathsf{sim} (\bv)$.
Then letting $\varphi: \mathcal{X} \rightarrow {\bf B}_2^1$ be the substitution that maps $x$ to $a$ and $y$ to $1$, we obtain the contradiction $\varphi(\bu) = a^k=0 \ne a=\varphi(\bv)$. Thus $\mathsf{non} (\bv) = \mathsf{non} (\bu)=\{x, y\}$.

Suppose that $\sh({\bv}) \ne \sh({\bu})$. Then $\sh({\bv})=y$. Let $\varphi: \mathcal{X}
\rightarrow {\bf B}_2^1$ be the substitution that maps $x$ to $a$ and $y$ to $b$. Then $\varphi(\bu) = ab$ but $\varphi(\bv) \in b\{ a, b\}^{+}=\{0, ba, b\}$, and so we obtain the contradiction $\varphi(\bu) \ne \varphi(\bv)$. Thus $\sh({\bv}) = \sh({\bu})$. By symmetry, we can show that $\st({\bv})=\st({\bu})$.

Suppose that $x^2$ is a subword of $\bv$. Then letting $\varphi: \mathcal{X}
\rightarrow {\bf B}_2^1$ be the substitution that maps $x$ to $a$ and $y$ to $b$, we obtain the contradiction $\varphi(\bu) = ab \ne 0 = \dots a^2 \ldots =\varphi(\bv)$. Thus $\bv$ does not have subword $x^2$. Similarly, $\bv$ does not have subword $y^2$. Therefore the only form of $\bv$ must be $\bv=(xy)^h$ for some $h \geq 2$.
\end{proof}

By a similar argument, we have the following lemma.
\begin{lem} \label{lem: B21 xyxyx}
Let $\bu \approx \bv$ be any identity satisfied by the
monoid ${\bf B}_2^1$ and $\bu =(xy)^kx$ for some $k \geq 2$. Then
$\bv=(xy)^hx$ for some $h \geq 2.$
\end{lem}

For each odd number $n\geq 3$, define
\[
\bw_n  = x_1 y x_2 z \dots x_n y x_1 z  x_2 y \dots y x_n.%
\]
By this definition, it is easy to see that ${\bw_n}_{\{x_1, x_2, \dots, x_n\}}=(x_1 x_2 \dots x_n)^2$ and ${\bw_n}_{\{y, z\}}=(yz)^{n-1}y$.

\begin{lem}\label{lem: B21 isoterm}
The word $\bw_n$ is an isoterm for ${\bf B}_2^1$.
\end{lem}

\begin{proof}
Suppose that $\bw_n \approx \bw'_n$ is an identity satisfied by the
monoid ${\bf B}_2^1$. Then it suffices to show that $\bw'_n = \bw_n$.

Since for each $i<j$, ${\bw_n}_{\{x_i, x_j\}} = (x_ix_j)^2$, it follows from Lemma~\ref{lem: B21 xyxy} that ${\bw'_n}_{\{x_i, x_j\}} = (x_ix_j)^h$ for some $h\geq 2$. From this it easily follows that
\begin{enumerate}
\item[$(\rm a)$] ${\bw'_n}_{\{x_1, x_2, \dots, x_n\}}=(x_1 x_2 \dots x_n)^h$ for some $h \geq 2$.
\end{enumerate}
Since ${\bw_n}_{\{y, z\}}=(yz)^{n-1}y$, it follows from Lemma~\ref{lem: B21 xyxyx} that
\begin{enumerate}
\item[$(\rm b)$] ${\bw'_n}_{\{y, z\}}=(yz)^\ell y$ for some $\ell \geq 2.$
\end{enumerate}

Suppose that $yz$ is a subword of $\bw'_n$. Then letting $\varphi: \mathcal{X}
\rightarrow {\bf B}_2^1$ be the substitution that maps $y$ and $z$ to $b$ and any other letter to $a$, we obtain the contradiction $\varphi(\bw_n) = ab\dots a=a \ne 0=\dots b^2 \dots =\varphi(\bw'_n)$. Thus
\begin{enumerate}
\item[$(\rm c)$] $\bw'_n$ does not have subword $yz$.
\end{enumerate}

Suppose that $x_ix_{i+1}$ is a subword of $\bw'$ (if $i=n$, then let $i+1=1$). Then letting $\varphi: \mathcal{X} \rightarrow {\bf B}_2^1$ be the substitution that maps $y$ and $z$ to $b$ and any other letter to $a$, we obtain the contradiction $\varphi(\bw_n)  = ab\dots a=a \ne 0 = \dots a^2 \dots =\varphi(\bw'_n)$. Thus
\begin{enumerate}
\item[$(\rm d)$]  $\bw'_n$ does not have subword $x_ix_{i+1}$.
\end{enumerate}

Suppose that $\sh({\bw_n}) \ne \sh({\bw'_n})$. Then it follows from $(\rm a)$ and $(\rm b)$ that $\sh({\bw'_n})= y$. Let $\varphi: \mathcal{X}
\rightarrow {\bf B}_2^1$ be the substitution that maps $y$ and $z$ to $b$ and any other letter to $a$.  Then $\varphi(\bw_n) = ab \dots  a=a$ but $\varphi(\bw'_n) \in b\{ a, b\}^{+}=\{0, ba, b\}$, and so we obtain the contradiction $\varphi(\bw_n) \ne \varphi(\bw'_n)$. Thus
\begin{enumerate}
\item[$(\rm e)$] $\sh({\bw_n}) = \sh({\bw'_n})=x_1$.
\end{enumerate}
By symmetry, we may show that
\begin{enumerate}
\item[$(\rm f)$] $\st({\bw_n}) = \st({\bw'_n})=x_n$.
\end{enumerate}

Now it follows from $(\rm a)-(\rm d)$ that $y$ (resp.~$z$) must be sandwiched between each $x_i$ and $x_{i+1}$, and each $x_i$, except for $\sh({\bw'_n})$ and $\st({\bw'_n})$, must be sandwiched between $y$ and $z$. Hence it follows from $(\rm a), (\rm e),(\rm f)$ that $\bw'_n$ is of the form
\begin{equation*}
\underbrace{\ x_1 y x_2 z \dots x_n y x_1 z  x_2 y \dots y x_n}_{\bw_n} \underbrace{z x_1 y  \dots y x_n}_{\bp} .%
\end{equation*}

Suppose that $\bp \ne \varnothing$. Let $\varphi:\mathcal{X} \rightarrow {\bf B}_2^1$ be the substitution given by
\[
t \mapsto %
\begin{cases}
a & \text{if } t \in \{x_1, x_3, \ldots , x_n \}, \\ %
1 & \text{if } t \in \{z, x_2, x_4, \ldots , x_{n-1} \}, \\ %
b & \text{if } t = y. %
\end{cases}
\]
Then $\varphi(\bw_n) = (a\cdot b\cdot 1 \cdot 1)^{\frac{n-1}{2}} a\cdot b \cdot (a\cdot 1\cdot 1 \cdot b)^{\frac{n-1}{2}}\cdot a=a$ but $\varphi(\bw'_n) = \varphi(\bw_n) \varphi(z) \varphi(x_1) \dots = a \cdot 1 \cdot a \dots =0$, and so we obtain the contradiction $\varphi(\bw_n) \ne \varphi(\bw'_n)$. Thus $\bp = \varnothing$, and so $\bw'_n = \bw_n$.
\end{proof}

\begin{lem}\label{lem: Z isoterm}
The word $\bw_n$ is an isoterm for the variety $\mathbb{B}_2^1\cap[\![ x^2y\approx yx^2]\!]$.
\end{lem}

\begin{proof}
Since $\bw_n$ is an isoterm for the variety ${\bf B}_2^1$, it suffices to show that any nontrivial identity $\bw_n \approx \bw'_n$ cannot be obtained from $x^2y\approx yx^2$. Suppose that there exist words $\be, \mathbf{f} \in \mathcal{X}^{*}$ and a substitution $\theta$ such that $\bw_n = \be ((x^2y) \theta) \mathbf{f}$
and $\bw'_n = \be \theta (yx^2) \mathbf{f}$, or $\bw_n = \be \theta(yx^2) \mathbf{f}$
and $\bw'_n = \be \theta(x^2y) \mathbf{f}$. By symmetry, we may assume that $\bw_n = \be \theta(x^2y) \mathbf{f}$ and $\bw'_n = \be \theta(yx^2) \mathbf{f}$. Since $x^2$ is a square subword in $x^2y$, we have $\theta(x^2)$ is a subword of $\bw_n$. But it is obvious that $\bw_n$ does not have any square subword, whence such $\theta$ does not exist. Hence any nontrivial identity $\bw_n \approx \bw'_n$ can not be obtained from $x^2y\approx yx^2$.
\end{proof}

For each odd number $n\geq 3$, define
\[
\bw'_n  = x_1 z x_2 y \dots x_n z x_1 y  x_2 z \dots z x_n.%
\]
It is easy to see that $\bw'_n $ can be obtained from $\bw_n $ by exchanging $y$ and $z$.

\begin{lem}\label{lem: Z sat}
For each odd number $n\geq 3$, the variety $\mathbb{M}(\bz_\infty)$ satisfies the identity $\bw_n \approx \bw'_n.$
\end{lem}

\begin{proof}
Since every subword of a Zimin word has a linear letter, every unavoidable word has a linear letter.
Hence if $\bw$ has no linear letters, then an assignment into $\mathbb{M}(\bz_\infty)$ that takes $\bw$ to a nonzero value must assign every letter to the value~$1$.
Thus $\mathbb{M}(\bz_\infty)$ satisfies any identity that formed by two words which have the same content and neither contains a linear letter.
In particular, $\mathbb{M}(\bz_\infty)$ satisfies the identity $\bw_n \approx \bw'_n$.
\end{proof}

\begin{thm}
The interval $[\mathbb{M}(\bz_\infty),\mathbb{B}_2^1\cap[\![ x^2y\approx yx^2]\!]]$ has cardinality $2^{\aleph_0}$.
\end{thm}

\begin{proof}
Let $\mathbb{O}$ be the odd natural numbers with each number is no less than $3$. Fix a subset $S$ of $\mathbb{O}$ and let $n$ be any number in $\mathbb{O}$. Let $\Sigma$ be a set of identities satisfied by the variety $\mathbb{B}_2^1\cap[\![ x^2y\approx yx^2]\!]$ and $\Sigma_S$ be the set $\Sigma \cup \{ \bw_i \approx \bw'_i: i\in S \}$. We will show that $\Sigma_S \vdash \bw_n \approx \bw'_n$ only if $n \in S$. That is for each pair of subsets $P, Q$ of $\mathbb{O}$, the sets of identities $\Sigma_{P}$ and $\Sigma_{Q}$ define the same subvariety of $V$ if and only if $P = Q$. Since there are uncountably many subsets of the odd natural numbers, there are uncountably many subvarieties in the interval $[\mathbb{M}(\bz_\infty),\mathbb{B}_2^1\cap[\![ x^2y\approx yx^2]\!]]$ by Lemma~\ref{lem: Z sat}.

Suppose that $\Sigma_{S} \vdash \bw_n \approx \bw'_n $. Then we can select
a number $m$ and pairwise distinct words $\bu_1,  \bu_2, \ldots, \bu_m$ with $\bu_1= \bw_n,
\bu_m= \bw'_n$, so that for each $i \leq  m$, there is a substitution $\theta_i$ and an identity $\bp_i \approx \bq_i \in \Sigma_{S}$ so that $\bu_{i+1}$ is obtained from $\bu_i$ by replacing a subword of the form $\theta_i(\bp_i)$ in $\bu_i$ with the subword $\theta_i(\bq_i)$.

First we consider the deduction from $\bu_1$ to $\bu_2.$  If $\bp_1 \approx \bq_1$ is an identity from $\Sigma$, then $\bu_1 =\bu_2$ by Lemma~\ref{lem: Z isoterm}. If $\bp_1 \approx \bq_1 \in \{ \bw_i \approx \bw'_i: i\in S \}$, then for some $j \in S$,  either $\bu_1=\be \theta_1(\bw_j) \mathbf{f}$ and $\bu_2=\be \theta_1(\bw'_j) \mathbf{f}$, or $\bu_1=\be \theta_1(\bw'_j) \mathbf{f}$ and $\bu_2=\be \theta_1(\bw_j) \mathbf{f}$. By symmetry, we may assume that
$\bu_1=\be \theta_1(\bw_j) \mathbf{f}$ and $\bu_2=\be \theta_1(\bw'_j) \mathbf{f}$.

As we are working in the lattice of monoid varieties, we need to consider substitutions of the form $\theta_1:\mathcal{X}^*\to\mathcal{X}^*$.    Note that if $\theta_1(y)=\theta_1(z)=1$ then $\theta_1(\bw_j)=\theta_1(\bw'_j)$ so that $\bu_1=\bu_2$.
We will find that all other possibilities lead to a contradiction.
Thus we assume, for contradiction, that at least one of $y$, $z$ is not assigned $1$ by $\theta_1$.
This implies, in particular, that $\theta_1(\bw_j)$ contains a (nonempty) subword that appears at least $j$ times, without overlap (namely, $\theta_1(y)$ or $\theta_1(z)$).
Thus we cannot have $j>n$, because in this case
 $\bu_1=\bw_n$ does not have a subword appearing $j$ times.
 Therefore we may assume that $j<n$.
 Since each of the letters $\{x_1, \ldots, x_j, y, z\}$ in $\bw_j$ occurs no less than twice in $\bw_j$, the image of each letter in $\{x_1, \ldots, x_j, y, z\}$ under $\theta_1$ is either $1$ or occurs no less than twice in $\bu_1=\bw_n$.
 But it is easy to see that any subword with length at least two occurs just once in $\bu_1=\bw_n$. Therefore
 \begin{equation}
 |\mathsf{con}(\theta_1(x_1))|,\dots,|\mathsf{con}(\theta_1(x_j))|,
|\mathsf{con}(\theta_1(y))|,|\mathsf{con}(\theta_1(z))|\leq 1.\tag{$*$}\label{eqn:1}
\end{equation}
We are assuming that at least one of $y,z$ is not assigned $1$; let us suppose $|\theta_1(y)|=1$, with the case of $|\theta_1(z)|=1$ by symmetry.   As $\theta_1(y)$ occurs no less than three times in $\bw_j$, we cannot have $\theta_1(y)=x_i$ for some~$i$, as $x_i$ appears just twice.  So $\theta_1(y)\in\{y,z\}$.  Consecutive instances of $y$ (and $z$, respectively) in $\bw_n$ are separated by subwords of length four.  As this is true in $\bw_j$ also, and  $\theta_1(y)\in\{y,z\}$, it follows from Equation~\eqref{eqn:1} that all letters appearing in $\bw_j$ between occurrences of $y$ are assigned letters by $\theta_1$ (that is, are not assigned $1$).  As all letters in $\con(\bw_j)$ have an occurrence between occurrences of $y$, it follows that $\{\theta_1(y),\theta_1(z)\}=\{y,z\}$ and $\theta_1(x_k) =x_i$ for some $i,k$.  However this leads to a contradiction because the length of subword $x_k \dots x_k$ of $\bw_j$ is $2j+1$, while the length of subword $x_i \dots x_i$ of $\bw_n$ is $2n+1$. Thus the substitution $\theta_1$ does not exist if $j <n$. Therefore in every case we have $\bu_2=\bu_1$. Consequently, $\bu_1=\bu_2=\dots=\bu_m$ and so $\Sigma_S \vdash \bw_n \approx \bw'_n$ only if $n \in S$, as required.
\end{proof}

Now we turn to the task of showing that $\mathbb{M}(\bz_\infty)$ is not defined by any finite set of identities within $\mathbb{B}_2^1$.

\begin{lem}\label{lem: Z many isoterms}
For each $n \geq 1$, the following words are the isoterms for $\mathbb{M}(\bz_\infty)$:
\begin{enumerate}[\rm(i)]
\item $a_1 \dots a_n,$
\item $a_1 \dots a_n s a_1 \dots a_n$,
\item $tasa$, $asat$,
\item $asbatb$,
\item $b^{\delta_1} a_1 b a_2 \dots  b a_n b^{\delta_2}$ for some $\delta_1,\delta_2 \in \{0, 1\}$, %
\item $b^{\delta_1} a_1 b a_2 \dots  b a_n b^{\delta_2} s b^{\delta_3} a_k b a_{k+1} \dots  b a_m b^{\delta_4}$ for some  $\delta_1, \ldots, \delta_4 \in \{0, 1\}$ and $1\leq n,k \leq m$. %
\item $b^{\delta_1} \bm  b^{\delta_2} a_0 b \bp b^{\delta_3} \bq b \bm b^{\delta_4} \bn b a_{n+1} b^{\delta_5} \bq b^{\delta_6}$ for some  $\delta_1, \ldots, \delta_6 \in \{0, 1\}$, $1\leq p <q <m \leq n$, $\bp =a_1 b \dots a_{p-1} b a_p$, $\bq= a_{p+1} b \dots  a_{q-1}b a_{q}$, $\bm = a_{q+1} b \dots a_{m-1} b a_m$ and $\bn = a_{m+1} b  \dots a_{n-1} b a_n$. %
\end{enumerate}
\end{lem}

\begin{proof}
(i) First it is easy to see that the words $a$ and $ab$ are isoterms for $\mathbb{M}(\bz_\infty)$. Thus the word $a_1 \dots a_n$ is an isoterm for $\mathbb{M}(\bz_\infty)$ by induction.

(ii) First we show that $absab$ is an isoterm for $\mathbb{M}(\bz_\infty)$. Let $\varphi: \mathcal{X} \rightarrow \mathcal{X}^{+}$ be the substitution that maps $a$ to $x_0$ and $b$ to $x_1x_0$ and $s$  to $x_2$. Then $\varphi(absab) = x_0x_1x_0 x_2 x_0x_1x_0 =\bz_{2}$, and so $absab$ is an isoterm for $\mathbb{M}(\bz_\infty)$. Thus the word $a_1 \dots a_n s a_1 \dots a_n$ is an isoterm for $\mathbb{M}(\bz_\infty)$ by induction.

(iii) Let $\varphi: \mathcal{X} \rightarrow \mathcal{X}^{+}$ be the substitution that maps $a$ to $x_0$ and $s$ to $x_1$ and $t$ to $x_2$. Then $\varphi(tasa) = x_2 x_0x_1x_0 $ and $\varphi(asat) = x_0x_1x_0 x_2 $. It is obvious that both of them are subwords of $\bz_{2}$. Thus $tasa$ and $asat$ are isoterms for $\mathbb{M}(\bz_\infty)$.

(iv) Let $\varphi: \mathcal{X} \rightarrow \mathcal{X}^{+}$ be the substitution that maps $a$ to $x_0$, $b$ to $x_1$, $s$ to $x_1x_0x_2x_0$ and $t$ to $x_3x_0$. Then $\varphi(asbatb) = x_0 x_1x_0x_2x_0x_1x_0x_3x_0x_1$, which is the subword of $\bz_{3}$. Thus $asbatb$ is an isoterm for $\mathbb{M}(\bz_\infty)$.

(v) Note that if a word is an isoterm, then all of its subwords are isoterms. Therefore we only need to show that the word $\bw =b a_1 b a_2 \dots  b a_n b$ is an isoterm, that is  $\delta_1=\delta_2=1$. Let $\varphi: \mathcal{X} \rightarrow \mathcal{X}^{+}$ be the substitution that maps $b$ to $x_0$, $a_{2^{i-1}+2^{i}k}$ to $x_i$ for each $k\in \mathbb{N}$, that is,
\[
a_i \mapsto %
\begin{cases}
x_{1} & \text{if } i =1+2k, \\ %
x_{2} & \text{if } i =2+4k, \\ %
x_{3} & \text{if } i =4+8k , \\ %
\ldots & \text{ } \ldots. %
\end{cases}
\]
Generality is not lost by assuming that $2^{\ell} \leq n < 2^{\ell+1}$ for some $\ell$. Then
$\varphi(\bw)=\bz$  is a prefix of $\bz_{\ell}$. Thus $\bw$ is an isoterm for $\mathbb{M}(\bz_\infty)$.

(vi) By the same arguments in (v) we may assume that $\delta_1=\delta_4=1$. In the following, we will show that the words can be mapped by a substitution into~$\bz_\infty$.
Without loss of generality, we may assume that $\delta_2 = \delta_3 = 1$ since if $\delta_2 =0$ (resp.~$\delta_3 =0$), then we can map $s$ to $bs$ (resp.~$sb$).
%
Therefore it suffices to show that
$\bw = b a_1 b a_2 \dots  b a_n b s b a_k b a_{k+1} \dots  b a_m b$
 is an isoterm for $\mathbb{M}(\bz_\infty)$. If $n<k$, then the result is true by (v). Therefore we may assume that $n \geq k$. Let $\varphi: \mathcal{X} \rightarrow \mathcal{X}^{+}$ be the substitution that maps $b$ to $x_0$,  and $a_{2^{i-1}+2^{i}k}$ to $x_i$ for each $k\in \mathbb{N}$. Then by~(v),
\begin{align*}
&\varphi(b a_1 b a_2 \dots ba_{k-1} ba_k \dots b a_n b)\\%
=& \varphi(b a_1 b a_2 \dots ba_{k-1}) \varphi (ba_k \dots b a_n b) \\%
=& AB
\end{align*}
is a prefix of $\bz_{\ell}$ where $\ell$ is the number such that $2^{\ell} \leq n < 2^{\ell+1}$. Write $\bz_\ell$ as $\bz_{\ell}= AB\bv_{\ell}$ and let $s$ map to $\bv_{\ell} x_{\ell+1} A$, and $a_{n+1}$ to $\bv_{\ell} x_{\ell+2}$, and for each $i \geq 3$, $a_{n-1+2^{i-1}+2^{i}k}$ to $\overline{\bz}_{\ell} x_{\ell+i}$,
\[
a_{n-1+i} \mapsto %
\begin{cases}
\overline{\bz}_{\ell} x_{l+1} & \text{if } i =1+2k, \\ %
\overline{\bz}_{\ell} x_{l+2} & \text{if } i =2+4k, \\ %
\overline{\bz}_{\ell} x_{l+3} & \text{if } i =4+8k , \\ %
\ldots & \text{ } \ldots %
\end{cases}
\]
where $x_0\overline{\bz}_{\ell} =\bz_{\ell}$. Then
\begin{align*}
\varphi(\bw) = &\varphi(b a_1 \dots ba_k \dots  b a_n b s b a_k  \dots  b a_n b a_{n+1}b \dots   b x_m b )\\%
=& \varphi(b a_1\dots ba_{k-1}) \cdot \varphi (ba_k \dots b a_n b) \cdot  \varphi(s) \cdot \varphi(b a_k \dots  b a_n b) \cdot \varphi(a_{n+1}) \cdot \varphi(bx_{n+2}) \dots \varphi(b x_m b ) \\%
=& \underbrace{ A \cdot B \cdot \bv_{\ell}}_{\bz_{\ell}} x_{\ell+1} \underbrace{A \cdot B \cdot\bv_{\ell}}_{\bz_{\ell}} x_{\ell+2} \cdot \bz_{\ell} x_{\ell+1} \dots \bz_{\ell+h} x_{\ell+h} x_0
\end{align*}
where $n-1+2^{h-1}+2^{h}k = m$ for some $k$. Hence $\varphi(\bw)$  is a subword of $\bz_{\infty}$. Therefore~$\bw$ is an isoterm for $\mathbb{M}(\bz_\infty)$.

(vii) By the same arguments in (v) we may assume that $\delta_1=\delta_6=1$.
In the following, we will show that the words can be mapped by a substitution into $\bz_\infty$. Without loss of generality, we may assume that $\delta_2 = \dots= \delta_5 = 1$ since if $\delta_2=0$ (resp.~$\delta_3=0, \delta_4=0, \delta_5 =0$), then we can map $a_0$ (resp.~$a_p, a_{m+1}, a_{n+1}$) to $ba_0$ (resp.~$a_pb, ba_{m+1}, a_{n+1}b$). Therefore it suffices to show that
\[
\bw=b~\bm~ b a_0 b ~\bp~ b ~\bq~ b ~\bm~ b~\bn~ b a_{n+1} b~\bq~b
\]
is an isoterm for $\mathbb{M}(\bz_\infty)$.
Let $\varphi: \mathcal{X} \rightarrow \mathcal{X}^{+}$ be the substitution that maps $b$ to $x_0$, $a_{2^{i-1}+2^{i}k}$ to $x_i$ for each $k\in \mathbb{N}$. Then by (v),
\begin{align*}
&\varphi(b~\bp~ b~\bq ~b ~\bm~ b ~\bn ~b )\\%
=& \varphi(b \bp) \varphi (b \bq) \varphi (b\bm) \varphi (b \bn b)\\%
=& ABCD
\end{align*}
is a prefix of $\bz_{\ell}$ where $\ell$ is the number such that $2^{\ell} \leq n < 2^{\ell+1}$. Let $\bz_{\ell}= ABCD \bv_{\ell}$ and  $a_0$ map to $D' \bv_{\ell}x_{\ell+1}$ where $x_0 D'= D$ and $a_{n+1}$ map to $\bv_{\ell}x_{\ell+2} A$ in $\varphi$. Then 
\[
\varphi(\bw)= \underbrace{C D  \bv_{\ell}}_{\mbox{the suffix of $\bz_\ell$}}x_{\ell+1} \underbrace{ABCD \bv_{\ell}}_{\bz_\ell}  x_{\ell+2} \underbrace{A x_0}_{\mbox{the prefix of $\bz_\ell$}}
\]
which is a subword of $\bz_{\ell+2}$. Therefore $\bw$ is an isoterm for $\mathbb{M}(\bz_\infty)$.
\end{proof}

\begin{lem}\label{lem:subword isoterm}
Any proper subword of $\bw_n$ is an isoterm for $\mathbb{M}(\bz_\infty)$.
\end{lem}

\begin{proof}
It is easy to see that any proper subword of $\bw_n$ must be a subword of one of the words
$\bw  =  y x_2 z \dots x_n y x_1 z  x_2 y \dots x_{n-1} y x_n$ and $\overline{\bw} = x_1 y x_2 z \dots x_n y x_1 z  x_2 y \dots x_{n-1} y$. Note that if a word is an isoterm, then all of its subwords are also isoterms. Therefore it suffices to show that both $\bw$ and $\overline{\bw}$ are isoterms for $\mathbb{M}(\bz_\infty)$. By symmetry, we will show that $\bw$ is an isoterm for $\mathbb{M}(\bz_\infty)$.

Suppose that $\mathbb{M}(\bz_\infty)$ satisfies some nontrivial identity $\bw \approx \bw'$.
Since
\[
\bw_{\{x_1,x_2, \ldots, x_n\}} = x_2 \dots x_n x_1 x_2 \dots x_n,
 \]
it follows from Lemma~\ref{lem: Z many isoterms} (ii) that $\bw_{\{x_1,x_2, \ldots, x_n\}}$ is an isoterm for  $\mathbb{M}(\bz_\infty)$. Therefore
\[
\bw'_{\{x_1,x_2, \ldots, x_n\}}= \bw_{\{x_1,x_2, \ldots, x_n\}} = x_2 \dots x_n x_1 x_2 \dots x_n,
\]
and so
\begin{enumerate}
\item[$(\rm a)$] $\bw'= X_1 x_2 X_2 x_3 X_3\dots x_n X_n x_1 X_{1'} x_2 X_{2'} \dots x_n X_{n'}$ for some $X_i, X_{i'} \in \{y,z\}^{*}$.
\end{enumerate}
Since
\[
\bw_{\{y, x_1, x_2, x_4, \ldots, x_{n-1}\}} = y x_2 y x_4 \dots y x_{n-1} y x_1 x_2 y \dots x_{n-1} y,
\]
it follows from Lemma~\ref{lem: Z many isoterms}(vi) that $\bw_{\{y, x_1, x_2, x_4, \ldots, x_{n-1}\}}$ is an isoterm. Therefore
\[
\bw'_{\{y,x_1, x_2, x_4, \ldots, x_{n-1}\}} =\bw_{\{y, x_1, x_2, x_4, \ldots, x_{n-1}\}} = y x_2 y x_4 \dots y x_{n-1} y x_1 x_2 y \dots x_{n-1} y,
\]
and so by $(\rm a)$, we have
\begin{enumerate}
\item[$(\rm b)$] $y$ is simple in each of $
X_1, X_2X_3, \ldots, X_{n-1}X_{n}, X_{2'}X_{4'}, \ldots, X_{(n-1)'}X_{n'}$ and $y$ does not occur in $X_{1'}$.
\end{enumerate}
Since
\[
\bw_{\{y, x_1, x_3, x_5, \ldots, x_{n}\}} = y x_3 y x_5 \dots y  x_{n} y x_1 y x_3 \dots x_{n-2} y x_{n},
\]
it follows from Lemma~\ref{lem: Z many isoterms}(vi) that $\bw_{\{y,x_3, x_5, \ldots, x_{n}\}}$ is an isoterm. Therefore
\[
\bw'_{\{y, x_1, x_3, x_5, \ldots, x_{n}\}} = \bw_{\{y, x_1, x_3, x_5, \ldots, x_{n}\}} = y x_3 y x_5 \dots y  x_{n} y x_1 y x_3 \dots x_{n-2} y x_{n}
\]
and so by $(\rm a)$, we have
\begin{enumerate}
\item[$(\rm c)$] $y$ is simple in each of $
X_1X_2, \ldots, X_{n-2} X_{n-1}, X_{n}, X_{1'}X_{2'}, \ldots, X_{(n-2)'}X_{(n-1)'}$ and $y$ does not occur in  $X_{n'}$.
\end{enumerate}
Now it follows from $(\rm b)$ and $(\rm c)$ that
\begin{enumerate}
\item[$(\rm d)$] $y$ is simple in each of $
X_1, X_3, \ldots, X_{n}, X_{2'}, \ldots, X_{(n-1)'}$ and $y$ does not occur in $
X_2, X_4, \ldots, X_{n-1}, X_{1'}, \ldots, X_{n'}$.
\end{enumerate}
By a similar argument we can show that
\begin{enumerate}
\item[$(\rm e)$] $z$ is simple in each of $
X_2, X_4, \ldots, X_{n-1}, X_{1'}, \ldots, X_{n'}$ and $z$ does not occur in $
X_1, X_3, \ldots, X_{n}, X_{2'}, \ldots, X_{(n-1)'}$.
\end{enumerate}
Therefore it follows from $(\rm a), (\rm d), (\rm e)$ that $X_1= X_3 = \dots= X_{n} = X_{2'}= \dots = X_{(n-1)'}= y$ and $X_2= X_4= \dots= X_{n-1}= X_{1'}= \dots = X_{n'} = z$, and so $\bw'=\bw$.
\end{proof}

\begin{lem}\label{lem:map isoterm}
A word obtained from $\bw_n$ by replacing some subword of length at least two by a linear letter is an isoterm for $\mathbb{M}(\bz_\infty)$.
\end{lem}

\begin{proof}
Let $\bw$ be a word obtained from $\bw_n$ by replacing some subword $\bp$ of length at least two by a linear letter $t$. Since the length of $\bp$ is at least two, some $x_i$ must occur in $\bp$. Suppose $\mathbb{M}(\bz_\infty)$ satisfies some nontrivial identity $\bw \approx \bw'$. Then there are three cases to be consider.

\noindent\textsc{Case~1}. The first $x_1$ of $\bw_n$ occurs in $\bp$ or the last $x_n$ of $\bw_n$ occurs in $\bp$. By symmetry, we may assume that the first $x_1$ occurs in $\bp$. Then
$\bw = t \bu$
where $\bu$ (possibly empty) is a suffix
of $\bw_n$. Since $\bu$ is a proper subword of $\bw_n$, it follows from Lemma~\ref{lem:subword isoterm} that $\bu$ is an isoterm. Therefore
\[
\bw'_{\mathsf{con}(\bu)} = \bw_{\mathsf{con}(\bu)} =\bu.
\]
Since $\bw_{\{t\}}=t$ and $t$ is an isoterm, it follows that $\bw'_{\{t\}}=t$.
If $\sh(\bu)=x_i$ and $x_i\in\simp(\bu)$ for some $i\geq 1$, then $\bw_{\{t, x_i\}}=tx_i$. It follows from Lemma~\ref{lem: Z many isoterms}(i) that $\bw_{\{t, x_i\}}$ is an isoterm. Therefore $\bw'_{\{t, x_i\}} = \bw_{\{t, x_i\}} = tx_i$, whence the occurrence of $t$ precedes $x_i$ in $\bw'$. Hence $\bw' = \bw$.

If $\sh(\bu)=x_i$ and $x_i\in\non(\bu)$ for some $i> 1$, then $\bw_{\{t, x_i, x_1\}} = tx_i x_1 x_i$. It follows from Lemma~\ref{lem: Z many isoterms}(iii) that $\bw_{\{t, x_i, x_1\}}$ is an isoterm. Therefore $\bw'_{\{t, x_i, x_1\}} = \bw_{\{t, x_i, x_1\}} = tx_i x_1 x_i$, whence the occurrence of $t$ precedes the first $x_i$ in $\bw'$. Hence $\bw' = \bw$.

If $\sh(\bu)=y$ or $\sh(\bu)=z$. By symmetry, we may assume that $\sh(\bu)=y$.
Since
\[
\bw_{\{t, y, x_1, x_3, \ldots, x_n\}} =%
\begin{cases}
t y x_i \dots y x_{n} & \text{if $x_1 \not\in \mathsf{con} (\bu)$}, \\ %
t y x_i \dots y x_{n} y x_{1} y x_3 \dots y x_{i} \dots yx_{n} & \text{if $x_1 \in \mathsf{con} (\bu)$}, \\ 
\end{cases}
\]
for appropriate odd number $i\geq 3$, it follows from Lemma~\ref{lem: Z many isoterms}(v) and (vi) that $\bw_{\{t, y, x_1, x_3, \ldots, x_n\}}$ is an isoterm. Therefore
\[
\bw'_{\{t, y, x_1, x_3, \ldots, x_n\}} = \bw_{\{t, y, x_1, x_3, \ldots, x_n\}},
\]
whence the occurrence of $t$ precedes the first $y$. Hence $\bw' = \bw$.

\noindent\textsc{Case~2}. Both the first $x_1$ and the last $x_n$ of $\bw_n$ not occur in $\bp$, and the second $x_1$ of $\bw_n$ occurs in $\bp$ or the first $x_n$ of $\bw_n$ occurs in $\bp$. By symmetry, we may assume that the second $x_1$ occurs in $\bp$. Then
\[
\bw = x_1 y x_2 z \dots x_j p^{\delta_1}  t  q^{\delta_2} x_k   \dots y x_n%
\]
for some $1\leq j \leq n$, $1< k \leq n$, $\delta_1, \delta_2 \in \{0, 1\}$ and
\[
p =%
\begin{cases}
y & \text{if $j$ is odd}, \\ %
z & \text{if $j$ is even}, \\ 
\end{cases}
\qquad
q =%
\begin{cases}
y & \text{if $k$ is odd}, \\ %
z & \text{if $k$ is even}. \\ 
\end{cases}
\]

\noindent\textsc{Subcase 2.1.}  $j \geq k$. Then $\bw_{\{t,x_k, \ldots, x_j\}} = x_k \dots x_j t x_k \dots x_j$, it follows from Lemma~\ref{lem: Z many isoterms}(ii) that $\bw_{\{t,x_k, \ldots, x_j\}}$ is an isoterm. Therefore
\begin{align} \label{xk-xj}
\bw'_{\{t,x_k, \ldots, x_j\}} =\bw_{\{t,x_k, \ldots, x_j\}} = x_k \dots x_j t x_k \dots x_j.
\end{align}
Since $\bw_{\{t, x_1, \ldots, x_{k-1}, x_k\}} = x_1 \dots x_{k-1} x_k tx_k$ (resp.~$\bw_{\{t, x_j,x_{j+1} \ldots, x_n\}} = x_j tx_j x_{j+1} \dots x_n$), it follows from Lemma~\ref{lem: Z many isoterms}(iii) that $\bw_{\{t, x_1, \ldots, x_{k-1}, x_k\}}$ (resp.~$\bw_{\{t, x_j,x_{j+1} \ldots, x_n\}}$) is an isoterm. Therefore
\begin{align}
\bw'_{\{t, x_1, \ldots, x_{k-1}, x_k\}} &= \bw_{\{t, x_1, \ldots, x_{k-1}, x_k\}} \!\!= x_1 \dots x_{k-1} x_k tx_k \label{x1-xk}\\%
\bw'_{\{t, x_j,x_{j+1} \ldots, x_n\}} &= \bw_{\{t, x_j,x_{j+1} \ldots, x_n\}} \!\! = x_j tx_j x_{j+1} \dots x_n \label{xj-xn}%
\end{align}
Now it follows from \eqref{xk-xj}, \eqref{x1-xk} and \eqref{xj-xn} that
\[
\bw'_{\{t, x_1, \ldots, x_k, \ldots, x_j, \ldots, x_n\}} = \bw_{\{t, x_1, \ldots, x_k,\ldots, x_j, \ldots, x_n\}} = x_1 \dots x_k \dots x_j t x_k \dots x_j \dots x_n.
\]
and so we have
\begin{enumerate}
\item[$(\rm a)$] $\bw' =  X_0 x_1 X_1 \dots x_j X_j t X x_k X_{k'} \dots x_n X_{n'}$ for some $X_i, X_{i'} \in \{y,z\}^{*}$.
\end{enumerate}
Since
\[
\bw_{\{t, y, x_2, x_4, \ldots, x_{n-1}\}} =%
\begin{cases}
y x_2 \dots  y x_{j-1} y^{\delta_1} t y^{\delta_2}x_{k+1} y \dots x_{n-1} y & \text{if $j$ and $k$ are odd}, \\ %
y x_2 \dots  y x_{j-1} y^{\delta_1} t x_{k} y \dots x_{n-1} y & \text{if $j$ is odd, $k$ is even}, \\ %
y x_2 \dots  y x_{j} t y^{\delta_2} x_{k+1} y \dots x_{n-1} y  & \text{if $j$ is even, $k$ is odd}, \\ %
y x_2 \dots  y x_{j} t x_{k} y \dots x_{n-1} y  & \text{if $j$ and $k$ are even}, \\ %
\end{cases}
\]
it follows from Lemma~\ref{lem: Z many isoterms}(vi) that $\bw_{\{t, y, x_2, x_4, \ldots, x_{n-1}\}}$ is an isoterm. Therefore
\[
\bw'_{\{t, y, x_2, x_4, \ldots, x_{n-1}\}} = \bw_{\{t, y, x_2, x_4, \ldots, x_{n-1}\}}
\]
and so by $(\rm a)$, we have
\begin{enumerate}
\item[$(\rm b_1)$] if $j$ and $k$ are odd, then $y$ is simple in  $X_0X_1, \ldots, X_{j-3}X_{j-2}, X_{(k+1)'}X_{(k+2)'},$ $ \ldots, X_{(n-1)'}X_{n'}$, and $y$ is simple in $X_{j-1}X_{j}$ (resp.~$XX_{k'}$) if and only if $\delta_1=1$ (resp.~$\delta_2=1$),

\item[$(\rm b_2)$] if $j$ is odd and $k$ is even, then $y$ is simple in $X_0X_1, \ldots, X_{j-3}X_{j-2}, X_{k'}X_{(k+1)'}, $ $ \ldots, X_{(n-1)'}X_{n'}$, and $y$ is simple in $X_{j-1}X_{j}$ if and only if $\delta_1=1$, $y$ does not occur in $X$,

\item[$(\rm b_3)$] if $j$ is even and $k$ is odd, then $y$ is simple in $X_0X_1, \ldots, X_{j-2}X_{j-1}, X_{(k+1)'}X_{(k+2)'},$ $ \ldots, X_{(n-1)'}X_{n'}$, and $y$ is simple in $XX_{k'}$ if and only if $\delta_2=1$, $y$ does not occur in $X_j$,

\item[$(\rm b_4)$] if $j$ and $k$ are even, then $y$ is simple in $X_0X_1, \ldots, X_{j-3}X_{j-2}, X_{k'}X_{(k+1)'} \ldots, $ $X_{(n-1)'}X_{n'}$, and $y$ does not occur in  $X_j$ and $X$.
\end{enumerate}
Since
\[
\bw_{\{t, y, x_1, x_3, \ldots, x_{n}\}} =%
\begin{cases}
x_1 y \dots  x_j y^{\delta_1} t y^{\delta_2} x_{k} \dots y x_{n} & \text{if $j$ and $k$ are odd}, \\ %
x_1 y \dots  x_j y^{\delta_1} t y x_{k+1} \dots y x_{n} & \text{if $j$ is odd, $k$ is even}, \\ %
x_1 y \dots  x_{j-1} y t y^{\delta_2} x_{k} \dots y x_{n}  & \text{if $j$ is even, $k$ is odd}, \\ %
x_1 y \dots  x_{j-1} y t  y x_{k+1} \dots y x_{n} & \text{if $j$ and $k$ are even}, \\ %
\end{cases}
\]
it follows from Lemma~\ref{lem: Z many isoterms}(vi) that $\bw_{\{t, y, x_1, x_3, \ldots, x_{n}\}}$ is an isoterm. Therefore
\[
\bw'_{\{t, y, x_1, x_3, \ldots, x_{n}\}} = \bw_{\{t, y, x_1, x_3, \ldots, x_{n}\}}
\]
and so by $(\rm a)$, we have
\begin{enumerate}
\item[$(\rm c_1)$] if $j$ and $k$ are odd, then $y$ is simple in  $X_1X_2, \ldots, X_{j-2}X_{j-1}, X_{k'}X_{(k+1)'},$ $ \ldots, X_{(n-2)'}X_{(n-1)'}$, and $y$ is simple in $X_{j}$ (resp.~$X$) if and only if $\delta_1=1$ (resp.~$\delta_2=1$), $y$ does not occur in $X_0$ and $X_n'$,

\item[$(\rm c_2)$] if $j$ is odd and $k$ is even, then $y$ is simple in $X_1X_2, \ldots, X_{j-2}X_{j-1}, XX_{k'}, $ $ \ldots, X_{(n-2)'}X_{(n-1)'}$, and $y$ is simple in $X_{j}$ if and only if $\delta_1=1$, $y$ does not occur in $X_0$ and $X_n'$,

\item[$(\rm c_3)$] if $j$ is even and $k$ is odd, then $y$ is simple in $X_1X_2, \ldots, X_{j-1}X_{j},  X_{k'}X_{(k+1)'},$ $ \ldots, X_{(n-2)'}X_{(n-1)'}$, and $y$ is simple in $X$ if and only if $\delta_2=1$, $y$ does not occur in $X_0$ and $X_n'$,

\item[$(\rm c_4)$] if $j$ and $k$ are even, then $y$ is simple in $X_1X_2, \ldots, X_{j-1}X_{j}, XX_{k'},  \ldots, $ $X_{(n-2)'}X_{(n-1)'}$, $y$ does not occur in $X_0$ and $X_n'$.
\end{enumerate}
Now it follows from $(\rm b)$ and $(\rm c)$ that
\begin{enumerate}
\item[$(\rm d_1)$] if $j$ and $k$ are odd, then $y$ is simple in  $X_1, X_3\ldots, X_{j-2}, X_{(k+1)'}, \ldots, X_{(n-1)'}$, and $y$ is simple in $X_{j}$ (resp.~$X$) if and only if $\delta_1=1$ (resp.~$\delta_2=1$), $y$ does not occur in $X_0, X_2, \ldots, X_{j-1}, X_{k'}, \ldots, X_{n'}$,

\item[$(\rm d_2)$] if $j$ is odd and $k$ is even, then $y$ is simple in $X_1, X_3, \ldots, X_{j-2}, X_{k'}, \ldots, X_{(n-1)'}$, and $y$ is simple in $X_{j}$ if and only if $\delta_1=1$, $y$ does not occur in $X_0, X_2, \ldots, X_{j-1}, $ $X, \ldots, X_{n'}$,

\item[$(\rm d_3)$] if $j$ is even and $k$ is odd, then $y$ is simple in $X_1, X_3, \ldots, X_{j-1}, X_{(k+1)'}, \ldots, $ $ X_{(n-1)'}$, and $y$ is simple in $X$ if and only if $\delta_2=1$, $y$ does not occur in $X_0, X_2, \ldots, X_{j},  X_{k'}, \ldots, X_{n'}$,

\item[$(\rm d_4)$] if $j$ is even and $k$ is odd, then $y$ is simple in $X_1, X_3, \ldots, X_{j-1}, X_{k'}, \ldots, X_{(n-1)'}$, and $y$ does not occur in $X_0, X_2, \ldots, X_{j},  X, \ldots, X_{n'}$.
\end{enumerate}
By a similar argument we can show that
\begin{enumerate}
\item[$(\rm e_1)$] if $j$ and $k$ are odd, then $z$ is simple in $X_0, X_2, \ldots, X_{j-1}, X_{k'}, \ldots, X_{n'}$,  and $z$ does not occur in $X_1, X_3\ldots, X_{j-2}, X_{(k+1)'}, \ldots, X_{(n-1)'}$,

\item[$(\rm e_2)$] if $j$ is odd and $k$ is even, then $z$ is simple in $X_0, X_2, \ldots, X_{j-1}, X, \ldots, X_{n'}$ and $z$ is simple in $X_j$ if and only if $\delta_1=1$, $z$ does not occur in $X_1, X_3, \ldots, $ $ X_{j-2}, X_{k'}, \ldots, X_{(n-1)'}$,

\item[$(\rm e_3)$] if $j$ is even and $k$ is odd, then $z$ is simple in $X_0, X_2, \ldots, X_{j},  X_{k'}, \ldots, X_{n'}$, and $z$ is simple in $X$ if and only if $\delta_2=1$, $z$ does not occur in $X_1, X_3, \ldots, $ $X_{j-1}, X_{(k+1)'}, \ldots, X_{(n-1)'}$,

\item[$(\rm e_4)$] if $j$ is even and $k$ is odd, then $z$ is simple in  $X_0, X_2, \ldots, X_{j},  X, \ldots, X_{n'}$, and $z$ is simple in $X_{j}$ (resp.~$X$) if and only if $\delta_1=1$ (resp.~$\delta_2=1$), $z$ does not occur in $X_1, X_3, \ldots, X_{j-1}, X_{k'}, \ldots, X_{(n-1)'}$.
\end{enumerate}
Therefore it follows from $(\rm a), (\rm d), (\rm e)$ that $\bw'=\bw$.

\noindent\textsc{Subcase 2.2.}  $j < k$. Then $\bw_{\{t, x_1, \ldots, x_j, x_k, \ldots, x_n\}} = x_1 \dots x_j t x_k \dots x_n$. It follows from Lemma~\ref{lem: Z many isoterms}(i) that $\bw_{\{t, x_1, \ldots, x_j, x_k, \ldots, x_n\}}$ is an isoterm. Therefore
\[
\bw'_{\{t, x_1, \ldots, x_j, x_k, \ldots, x_n\}} = \bw_{\{t, x_1, \ldots, x_j, x_k, \ldots, x_n\}} = x_1 \dots x_j t x_k \dots x_n.
\]
Now repeat the same argument with Case 2.1 except replacing Lemma~\ref{lem: Z many isoterms}(vi) by Lemma~\ref{lem: Z many isoterms}(v), we can show that $\bw'=\bw$.

\noindent\textsc{Case~3}. $x_i$ in $\bp$ for some $1<i<n$ and $x_1, x_n $ not in $\bp$. Then either the first $x_i$ occurs in $\bp$ or the second $x_i$ occurs in $\bp$. By symmetry, we may assume that the first $x_i$ occurs in $\bp$. Without loss of generality, we may assume that
\[
\bw = x_1 y x_2 z \dots x_j p^{\delta_1}  t  q^{\delta_2} x_k \dots z x_n y x_1 z \dots x_j \dots x_i \dots x_k \dots x_{n-1} y x_n %
\]
for some $1< j\leq i \leq k  \leq n$, $j \ne k$,  $\delta_1, \delta_2 \in \{0, 1\}$ and
\[
p =%
\begin{cases}
y & \text{if $j$ is odd}, \\ %
z & \text{if $j$ is even}, \\ 
\end{cases}
\qquad
q =%
\begin{cases}
z & \text{if $k$ is odd}, \\ %
y & \text{if $k$ is even}. \\ 
\end{cases}
\]
Since $\bw_{\{t, x_1, \ldots, x_j\}} = (x_1 \dots x_j) t  (x_1 \dots x_j)$ (resp.~$\bw_{\{x_{j+1}, \ldots, x_n\}} = (x_k \dots x_n) x_{j+1}$ $ \dots x_{k-1} (x_k \dots x_n)$),
it follows from Lemma~\ref{lem: Z many isoterms}(ii) that $\bw_{\{t, x_1, \ldots, x_j\}}$ (resp.~$\bw_{\{x_{j+1}, \ldots, x_n\}}$) is an isoterm. Therefore
\begin{align}
\bw'_{\{t, x_1, \ldots, x_j\}} &= \bw_{\{t, x_1, \ldots, x_j\}} = (x_1 \dots x_j) t  (x_1 \dots x_j)  \label{t x1-xj}\\%
\bw'_{\{x_{j+1}, \ldots, x_n\}} &= \bw_{\{x_{j+1}, \ldots, x_n\}} = (x_k \dots x_n) x_{j+1} \dots x_{k-1} (x_k \dots x_n). \label{xi xk-xn}  %
\end{align}
Since $\bw_{\{t, x_1, x_{j+1}, x_n\}} = x_1 t x_n x_1 x_{j+1}x_n$, it follows from Lemma~\ref{lem: Z many isoterms}(iv) that $\bw_{\{t, x_1, x_{j+1}, x_n\}}$ is an isoterm. Therefore
\begin{align}\label{x1xn}
\bw'_{\{t, x_1, x_{j+1}, x_n\}} = \bw_{\{t, x_1, x_{j+1}, x_n\}} = x_1 t x_n x_1 x_{j+1}x_n,
\end{align}
that is, $t$ precedes $x_{j+1}$ in $\bw'$ and the first $x_n$ precedes the second $x_1$ in $\bw'$.
Now it follows from \eqref{t x1-xj}, \eqref{xi xk-xn} and \eqref{x1xn} that
\[
\bw'_{\{t, x_1, \ldots, x_n\}} = (x_1 \dots x_j) t (x_k \dots x_n) (x_1 \dots x_j)  x_{j+1} \dots x_{k-1} (x_k \dots x_n), %
\]
and so
\begin{enumerate}
\item[$(\rm a)$] $\bw' =  X_0 x_1 X_1 \dots x_j X_j t X x_k X_{k} \dots x_n X_{n} x_1 X_{1'} \dots x_n X_{n'}$ for some $X_i, X_{i'} \in \{y,z\}^{*}$.
\end{enumerate}
Now repeat the same argument with Case 2.1 except replacing Lemma~\ref{lem: Z many isoterms}(vi) by Lemma~\ref{lem: Z many isoterms}(vii), we can show that $\bw'=\bw$.
\end{proof}

\begin{thm}\label{thm: B to Z non-F}
The variety $\mathbb{M}(\bz_\infty)$ cannot be defined within $\mathbb{B}_2^1$ by any finite system of identities.
\end{thm}

\begin{proof}
Assume, for contradiction, that some finite system of identities defines  $\mathbb{M}(\bz_\infty)$ within $\mathbb{B}_2^1$, and let $m$ be the maximum number of variables appearing in any word in this basis.  Now let $n$ be larger than $m$ and consider a deduction of $\bw_n=\bw'_n$.  Now, the first nontautological step of such a deduction involves a substitution from some $m$-letter word into $\bw_n$.  By Lemma~\ref{lem:subword isoterm}, it must map onto all of $\bw_n$.  But as the number of variables is at most $m<n$, there exists some letter $z$ such that the factor $z \varphi$ contains at least two distinct letters, which is the case as described in Lemma~\ref{lem:map isoterm}.  Hence in either case, the image of this word is an isoterm for $M(\bz_\infty)$, contradicting the fact that this was a nontautological step of the deduction.
\end{proof}


\begin{thebibliography}{99}
\bibitem{ELL} C.C. Edmunds, E.W.H. Lee and K.W.K. Lee, Small semigroups generating varieties with continuum many subvarieties, Order 27 (2010), 83--100.
\bibitem{erdhaj} P. Erd\H{o}s and A. Hajnal, On chromatic number of graphs and set-systems, Acta Mathematica Academiae Scientiarum Hungaricae Tomus 17 (1966), 61--99.
\bibitem{fedvar} T. Feder and M.Y. Vardi, The computational structure of monotone monadic SNP and constraint satisfaction: a study through Datalog and group theory, SIAM J. Comput. 28 (1998), 57--104.
\bibitem{guslee} S.V. Gusev and E.W.H. Lee, Varieties of monoids with complex lattices of subvarieties, Bull. London Math. Soc. 52 (2020), 762--775.
\bibitem{gusleever} S.V. Gusev, E.W.H. Lee and B.M. Vernikov, The lattice of varieties of monoids, manuscript (2020), arXiv:2011.03679.
\bibitem{gusver} S.V. Gusev and B.M. Vernikov, Chain varieties of monoids, Dissertationes Math. 534 (2018), 1--73.
\bibitem{hamjac} L. Ham and M. Jackson, Axiomatisability and hardness for universal Horn classes of hypergraphs, Algebra Universalis 79 (2018), 1--17.
\bibitem{jac00} M. Jackson, Finite semigroups whose varieties have uncountably many subvarieties, J. Algebra 228 (2000), 512--535.
\bibitem{jac:INFB} M. Jackson, Small inherently nonfinitely based finite semigroups, Semigroup Forum 64 (2002), 297--324.
\bibitem{jac05} M. Jackson, Finiteness Properties of Varieties and the Restriction to Finite Algebras, Semigroup Forum 70 (2005), 159--187.
\bibitem{jck} M. Jackson, Flexible constraint satisfiability and a problem in semigroup theory, manuscript 2015--2020, arXiv:1512.03127.
\bibitem{jaclee} M. Jackson and E.W.H. Lee, Monoid varieties with extreme properties, Trans. Amer. Math. Soc. 370 (2018), 4785--4812.
\bibitem{jckmck}  M. Jackson and R. McKenzie, Interpreting graph colorability in finite semigroups, Internat. J. Algebra Comput. 16 (2006), 119--140.
\bibitem{kli} O. Klima, Complexity issues regarding checking of identities in finite monoids, Semigroup
Forum 79 (2009), 435--444.
\bibitem{lee} E.W.H. Lee, Varieties generated by $2$-testable monoids, Studia Scientiarum Mathematicarum Hungarica 49 (2012), 366--389.
\bibitem{leevol} E.W.H. Lee and M.V. Volkov, On the structure of the lattice of combinatorial Rees-Sushkevich varieties, in Semigroups and Formal Languages, World Scientific 2007, pp. 164--187.
\bibitem{leezha} W.T. Zhang and E.W.H. Lee, Finite basis problem for semigroups of order six, LMS J. Comput. Math. 18 (2015), 1--129.
\bibitem{may} P. Mayr,  On finitely related semigroups. Semigroup Forum 86 (2013), 613--633.
\bibitem{per} P. Perkins, Bases for equational theories of semigroups, J. Algebra 11 (1969), 298--314.
\bibitem{sap} M. V. Sapir, Problems of Burnside type and the finite basis property in varieties of semigroups (Russian), Izv. Akad. Nauk SSSR Ser. Mat. 51 (1987), no. 2, 319--340, 447; English transl., Math. USSR-Izv. 30 (1988), no. 2, 295--314.
\bibitem{sap2} M. V. Sapir, Inherently non-finitely based finite semigroups (Russian), Mat. Sb. (N.S.) 133(175) (1987), no. 2, 154--166, 270; English transl., Math. USSR-Sb. 61 (1988), no. 1, 155--166.
\bibitem{sei} S. Seif, The Perkins semigroup has co-NP-Complete term-equivalence problem, Internat. J. Algebra Comput.  15 (2005), 317--326.
\bibitem{tra} A.N. Trahtman, An identity basis of the five-element Brandt semigroup, Ural. Gos. Univ. Mat. Zap. 12 (1981), 147--149 [Russian].
\end{thebibliography}
\end{document}